\documentclass[
  a4paper,             
  11pt,                
  twoside,
  final,               
]{amsart}

\usepackage[sectionthm]{stephenpack}

\usepackage[
  hmarginratio={1:1},  
  textwidth=400pt,     
  textheight=672pt,    
  tmargin=3.0cm,       
  marginpar=3cm,
  marginparsep=.25cm,
  heightrounded,
  asymmetric,          
]{geometry}

\usepackage{subcaption}
\usepackage{multicol}
\usepackage{picins}
\usepackage{booktabs}

\usepackage[boxsize=1.5em]{ytableau}
\NewDocumentCommand \boxsize {} {1.5em}

\usepackage{tikz-cd}

\RenewDocumentCommand \cA     {}     {\mathcal{A}}
\NewDocumentCommand   \cAh    {}     {\mathcal{A}_\hbar}
\NewDocumentCommand   \cI     {}     {\mathcal{I}}
\NewDocumentCommand   \Acoset {O{}g} {\cA_\hbar \Phi_{#1}(\cI_{#2})}
\NewDocumentCommand   \Uh     {}     {U_\hbar}
\NewDocumentCommand   \tfm    {}     {\widetilde{\fm}}
\RenewDocumentCommand \Whit   {}     {\mathsfup{Wh}}
\NewDocumentCommand   \Weyl   { m }  {\mathbf{A}_{#1}\hspace{0pt}}
\NewDocumentCommand   \up     {mm}   {\prescript{}{#1}{\uparrow}^{#2}}

\NewDocumentCommand   \Aadj   {ommo}
  {\underset{\IfValueTF{#1}{#1 \to #2}{#3 \to #4}\vphantom{k}}{A_{#2 #3}}}
\NewDocumentCommand   \liez   { og }
  {\IfValueTF{#1}{\lie{z}_{#1}(#2)}{\lie{z}(#2)}}
\RenewDocumentCommand \vS     {t~}
  {\IfBooleanTF{#1}{\widetilde{\mathscr{S}}}{\mathscr{S}}}
\NewDocumentCommand   \ul     {mo}
  {\IfNoValueTF{#2} {\underline{#1}} {\underline{#1}_{\raisebox{0.5ex}{$\scriptstyle #2$}}}}

\NewDocumentCommand \W   {}     {W\nb}
\NewDocumentCommand \hsl {O{2}} {$\fsl_{#1}$\nb}

\crefname{property}{property}{properties}
\crefname{condition}{condition}{conditions}

\title[Quantum Hamiltonian reduction of W-algebras and
category~\texorpdfstring{$\cO$}{O}]
{Quantum Hamiltonian reduction of W-algebras \\
and category~\texorpdfstring{$\mathbfcal{O}$}{O}}

\author{Stephen Morgan}
\address{Mathematical Sciences Institute, Australian National University,
Acton ACT 2061, Australia}
\email{\emailurl{smorgan@math.utoronto.ca}}
\urladdr{\httpurl{www.math.utoronto.ca/smorgan}}
\thanks{This material is partially based upon work supported by the National
Science Foundation under grant no.~0932078~000 while the author was in residence
at the Mathematical Sciences Research Institute in Berkeley, California during the Fall 2014 semester.}

\begin{document}

\begin{abstract}
We define a quantum version of Hamiltonian reduction by stages, producing a
construction in type~A for a quantum Hamiltonian reduction from the
\W-algebra $U(\fg,e_1)$ to an algebra conjecturally isomorphic to $U(\fg,e_2)$,
whenever $e_2 \ge e_1$ in the dominance ordering.
This isomorphism is shown to hold whenever $e_1$ is subregular, and in $\fsl_n$
for all $n \le 4$.

We next define embeddings of various categories~$\cO$ for the \W-algebras
associated to $e_1$ and $e_2$, amongst them the embeddings
$\cO(e_2,\fp) \hookrightarrow \cO(e_1,\fp)$, where $\fp$ is a parabolic
subalgebra containing both $e_1$ and $e_2$ in its Levi subalgebra.
\end{abstract}

\maketitle

\section{Introduction}

Let $\fg$ be a semi-simple complex Lie algebra with universal enveloping algebra
$U(\fg)$, and let $G$ be the simply-connected algebraic group satisfying
$\Lie(G) = \fg$.
To any nilpotent element $e \in \fg$, one can associate a certain non-commutative
algebra $U(\fg,e)$ known as the \emph{\W-algebra} associated to the
nilpotent~$e$.
There are several definitions of the \W-algebra depending on different choices
and parameters, but it is known that they are all equivalent up to isomorphism,
and depend only on the nilpotent orbit of~$e$ under the adjoint action of~$G$.

The definition of \W-algebras of primary use in this paper is as a \emph{quantum
Hamiltonian reduction} of the universal enveloping algebra $U(\fg)$ with respect
to a choice of nilpotent subalgebra~$\fm$ and character~$\chi$ thereof, both
derived from the nilpotent~$e$.
In short, given a Lie algebra $\fm$ coming from a good grading of~$\fg$ and the
character $\chi \in \fm^*$ associated to~$e$ under the identification $\fg \simeq \fg^*$
given by the Killing form.
Considering the shift $\fm_\chi \coloneqq \set{y - \chi(y) \st y \in \fm}$, the \W-algebra can
be defined as the algebra of invariants in the quotient $U(\fg) \big/ U(\fg)
\fm_\chi$ under the adjoint action of~$\fm$;
i.e.~$U(\fg,e) \coloneqq \smash{\pp[\big]{U(\fg) \big/ U(\fg) \fm_\chi}^\fm}$.
Since all objects involved are filtered by the Kazhdan filtration, the
\W-algebra is itself a filtered algebra.
Taking the associated graded algebra yields the ring of functions on the Slodowy
slice $\vS_\chi \subseteq \fg^*$, and quantum Hamiltonian reduction reduces to ordinary
Hamiltonian reduction of Slodowy slices \cite{GG:QuantSlod}.

With this framework in mind, one can ask whether this quantum Hamiltonian
reduction can be decomposed into steps, analogous to the classical construction
of \emph{Hamiltonian reduction by stages}.
In particular, given a pair of \W-algebras defined by quantum Hamiltonian
reduction, when can an intermediate reduction between the two be found which
commutes with the original reductions up to isomorphism, as in
\cref*{fig:IntRedDiag}.

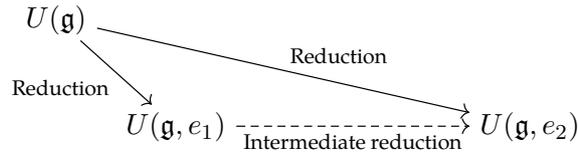
\begin{figure}[th!]
\begin{equation*}
\begin{tikzcd}[column sep=tiny]
U(\fg)
  \ar{dr}[swap]{\text{Reduction}}
  \ar{drrr}{\text{Reduction}} \\
 & U(\fg,e_1) \ar[dashed]{rr}[swap]{\text{Intermediate reduction}}
 &\hspace{6em} & U(\fg,e_2)
\end{tikzcd}
\end{equation*}
\caption{Reduction of W-algebras by stages.}
\label{fig:IntRedDiag}
\end{figure}

In \cref{chap:TypeA}, we give a partial answer to this question in type~A.
We first present a construction using the combinatorics of \emph{pyramids},
which for any pair of nilpotent elements $e_1, e_2 \in \fsl_n$, where $e_2$ covers
$e_1$ in the dominance ordering, produces an intermediate reduction from the
\W-algebra $U(\fg,e_1)$ in type~A to a certain algebra associated to $e_2$.
We conjecture that this algebra is isomorphic to the \W-algebra $U(\fg,e_2)$
(\cref{conj:typeAred}), and present a proof in known cases.

In \cref{chap:CatOWalg}, we turn our attention to the representation theory of
\W-algebras.
This is a subject which has been widely studied, and a number of connections to
the representation theory of the Lie algebra $\fg$ itself have been found.
The construction of quantum Hamiltonian reduction by stages produces a
$(U(\fg,e_1), U(\fg,e_2))$\nb-bimodule for any pair of nilpotents~$e_1$ and
$e_2$ as above.
This in turn provides a pair of adjoint functors
$\Mod{U(\fg,e_1)} \leftrightarrows \Mod{U(\fg,e_2)}$ for any such pair.
A modification of an argument of Loseu \cite{Los:CatOWAlg} can be used to
produce embeddings of the corresponding categories
$\cO(e_2,\fp) \hookrightarrow \cO(e_1,\fp)$, whenever $e_1 \le e_2$.


%

\subsection*{Acknowledgements}
I would like to thank my former PhD adviser J.~Kamnitzer for suggesting this
problem and providing many useful ideas.
I would also like to thank C.~Dodd for many fruitful discussions, and
R.~Bezrukavnikov, Zs.~Dancso, M.~Gualtieri, L.~Jeffrey, Y.~Karshon, A.~Licata,
I.~Loseu, and V.~Toledano-Laredo for comments and suggestions.

\section{W-algebras and quantum Hamiltonian reduction}

We first recall the definition of \W-algebras via quantum Hamiltonian reduction.
Let $\fg$ be a semisimple Lie algebra over $\CC$, and let $e \in \fg$ be a chosen
nilpotent element.
By the Jacobson--Morozov theorem any non-zero $e$ be completed to an \hsl-triple
$(e,h,f)$.
The semi\-simple element~$h$ determines a $\ZZ$\nb-grading of the Lie
algebra~$\fg$ by declaring $\fg = \bigoplus_{j \in \ZZ} \fg_j$, where
$\fg_j = \set{x \in \fg \st [h,x] = j x}$.
This grading satisfies the following list of useful properties, where
$\liez[\fg]{e}$ is the centraliser of $e$ in $\fg$, and $\killing{\cdot}{\cdot}$
is the Killing form:
\setlength{\multicolsep}{1ex}
\begin{enumerate}[
  label=(\textsf{GG}\arabic*),
  ref=\textsf{GG}\arabic*,
  labelindent=\parindent,
  leftmargin=*]
\item \label[property]{GGprop1}
  $e \in \fg_2$,
\item \label[property]{GGprop2}
  $\ad e \colon \fg_j \to \fg_{j+2}$ is injective for $j \le -1$,
\item \label[property]{GGprop3}
  $\ad e \colon \fg_j \to \fg_{j+2}$ is surjective for $j \ge -1$,
\item \label[property]{GGprop4}
  $\liez[\fg]{e} \subseteq \bigoplus_{j \ge 0} \fg_j$,
\item \label[property]{GGprop5}
  $\killing{\fg_i}{\fg_j} = 0$ unless $i+j=0$,
\item \label[property]{GGprop6}
  $\dim \liez[\fg]{e} = \dim \fg_0 + \dim \fg_1$.
\end{enumerate}

It is a well-known result that any grading which satisfies
\crefrange*{GGprop1}{GGprop3} will necessarily satisfy all of them
(and even more strongly that \cref*{GGprop2,GGprop3} are equivalent for any
$\ZZ$\nb-grading).
This motivates the following \namecref{def:GG}, which provides a generalisation
of the gradings coming from \hsl-triples.

\begin{defn} \label{def:GG}
A $\ZZ$\nb-grading of~$\fg$ is called a \emph{good grading} for the nilpotent $e$
if it satisfies \crefrange*{GGprop1}{GGprop3}.
A good grading which comes from an \hsl-triple containing~$e$ is called a
\emph{Dynkin grading}.
A good grading which vanishes in odd degree is called an \emph{even grading}.
\end{defn}

Note that although all Dynkin gradings are good, there exist good gradings which
are not Dynkin: these non-Dynkin good gradings will be important for our work.
From this point on, fix a good grading of the Lie algebra~$\fg$.

The space $\fg_{-1}$ has a natural symplectic form $\omega$ given by
$\omega(x,y) \coloneqq \killing{e}{[x,y]}$.
Choosing a Lagrangian subspace $\fl \subseteq \fg_{-1}$ with respect to this form, one
can define a \emph{Premet subalgebra}
$\fm \coloneqq \fl \oplus \smash{\bigoplus_{j \le -2}} \fg_j$.
Premet subalgebras enjoy a number of properties we record for future reference.
Let $\vO_e$ be the adjoint orbit through $e$.
\begin{enumerate}[
  label=($\chi$\arabic*),
  ref=$\chi$\arabic*,
  labelindent=\parindent,
  leftmargin=*]
\item \label[property]{PremNil}
    $\fm$ is an ad-nilpotent subalgebra of~$\fg$,
\item \label[property]{PremDim}
    $\dim \fm = \frac{1}{2} \dim \vO_e$,
\item \label[property]{PremZ}
    $\fm \cap \liez[\fg]{e} = 0$,
\item \label[property]{PremChar}
    $\chi \coloneqq \killing{e}{\cdot}$ restricts to a character of~$\fm$,
\end{enumerate}

Given a Lie algebra~$\fm$ with character~$\chi$, one can define the shifted Lie
algebra $\fm_\chi \coloneqq \set{y - \chi(y) \st y \in \fm}$.
With this in hand, we can define the \W-algebra.

\begin{defn} \label{def:Walg}
Let $e \in \fg$ be a nilpotent element with a chosen good grading and Lagrangian
subspace $\fl \subseteq \fg_{-1}$, and let $\fm$ be the associated Premet subalgebra.
The \emph{(finite) \W-algebra} $U(\fg,e)$ is the set of invariants in the
quotient $U(\fg) \big/ U(\fg) \fm_\chi$, under the adjoint action of $\fm$:
\begin{equation*}
U(\fg,e) 
  \coloneqq \pp[\big]{U(\fg) \big/ U(\fg)\fm_\chi}^\fm
  = \set[\big]{\overline{u} \in U(\fg) \big/ U(\fg)\fm_\chi
      \st [a,u] \in U(\fg)\fm_\chi \;\;\forall \;a \in \fm }.
\end{equation*}
\end{defn}

\subsection{Slodowy slices}
\label{sec:SlodowySlices}

For any nilpotent element $e \in \fg$, one can construct an \hsl-triple
$(e,h,f)$ by the Jacobson--Morozov theorem.
In fact, any such pair of triples $(e,h,f)$ and $(e',h',f')$ for which
$\vO_e = \vO_{e'}$ are conjugate by the adjoint action, and so in particular one
can speak unambiguously of ``the \hsl-triple associated to~$e$”.
If one additionally has a good grading~$\Gamma$ for the nilpotent~$e$, the \hsl-triple
can be chosen to be $\Gamma$\nb-graded, so that $e$, $h$ and~$f$ lie in graded
degrees 2, 0 and~-2, respectively.

Associated to an \hsl-triple $(e,h,f)$ is a certain subvariety $\vS_e \subseteq \fg$
known as the \emph{Slodowy slice}.
It is an affine space which forms a transverse slice to the nilpotent orbit
$\vO_e$, and is defined as a translate of the centraliser of $f$.
We shall usually deal with the Slodowy slice in the dual Lie algebra~$\fg^*$,
transported via the Killing isomorphism $\kappa \colon \fg \isoto \fg^*$.
\begin{align*}
\vS_e & \coloneqq e + \liez[\fg]{f} \subseteq \fg   &
\vS_\chi & \coloneqq \chi + \pp[\big]{\fg / [\fg,f]}^* = \kappa(\vS_e) \subseteq \fg^*
\end{align*}
The Slodowy slice $\vS_\chi$, and hence $\vS_e$, inherits a natural Poisson
structure from the variety~$\fg^*$ equipped with the Lie--Poisson bracket.

Slodowy slices are of independent interest, but for our purposes we consentrate
on their relation to \W-algebras.
Gan and Ginzburg proved in \cite{GG:QuantSlod} that $U(\fg,e)$ has the structure
of a filtered algebra, and that the corresponding associated graded algebra
$\gr U(\fg,e)$ is the ring of functions on the Slodowy slice
$\Cpoly[\big]{\vS_\chi}$.
Further, taking $M$ to be the algebraic group with Lie algebra~$\fm$, one can
consider the moment map of the co-adjoint action of $M$ on $\fg^*$:
this is just restriction map $\mu \colon \fg^* \to \fm^*$.
The Slodowy slice $\vS_\chi$ can be expressed as a Hamiltonian reduction
of~$\fg^*$:
\begin{equation*}
\vS_\chi \simeq \fg^* \qq{\chi} M \coloneqq \mu^{-1}(\chi) / M.
\end{equation*}
Expressed in terms of rings of functions, and taking $I_\chi$ to be the ideal of
functions which vanish on $\mu^{-1}(\chi)$, this takes the form
\begin{equation}
\Cpoly[\big]{\vS_\chi} \simeq \pp[\big]{\Cpoly{\fg^*} \big/ I_\chi}^M.
\label{eq:SlodowyRed}
\end{equation}

The Poisson structure which $\vS_\chi$ inherits as a Hamiltonian reduction of
$\fg^*$ agrees with the Poisson structure it inherits as a subvariety of $\fg^*$
(cf.~\cite[Section~3.4]{GG:QuantSlod}).
The similarity between \cref{def:Walg,eq:SlodowyRed} has led to \cref{def:Walg}
to be referred to as a \emph{quantum Hamiltonian reduction}, by analogy.
This can be formalised in the language of deformation quantisations.

\subsection{Deformation quantisations and quantum Hamiltonian reduction}

\begin{defn}
\label{def:DefQuant}
Let $A$ be a Poisson algebra with Poisson bracket $\poibr{\cdot,\cdot}$.
A \emph{deformation quantisation} of $A$ is an associative unital product
$\star \colon A \otimes A \to \power{A}{\hbar}$ such that, when extended
$\Cpower{\hbar}$-bilinearly, satisfies the following conditions:
\begin{enumerate}
\item \label[condition]{cond:DefQuantProd}
      $\star$ is an associative binary product on $\power{A}{\hbar}$, continuous
      in the $\hbar$-adic topology;
\item \label[condition]{cond:DefQuantMult}
      $f \star g = fg + O(\hbar)$ for all $f,g \in A$;
\item \label[condition]{cond:DefQuantPoisson}
      $f \star g - g \star f = \poibr{f,g}\hbar + O(\hbar^2)$ for all $f,g \in A$.
\end{enumerate}
Writing $f \star g = \sum_{k\ge0} D_k(f,g) \hbar^{k}$, we shall further require that
$\star$ be a \emph{differential} deformation quantisation, that is one satisfying
the additional condition:
\begin{enumerate}
\setcounter{enumi}{3}
\item for each $k$, $D_k(\cdot,\cdot)$ is a bidifferential operator of order at
      most $k$ in each variable.
\end{enumerate}

The vector space $\power{A}{\hbar}$ equipped with the multiplication $\star$
shall be denoted $\cAh$, and is often referred to as a deformation
quantisation itself.
If~$X$ is a Poisson variety and~$A$ is its ring of functions, we often say that
$\cAh$ is a deformation quantisation of~$X$.
\end{defn}

The product $\star$ can also be used to introduce a new associative product on
$A$ through the projection $\power{A}{\hbar} \to A$, given by sending $\hbar$
to~$1$.
More concretely, define the product $\circ \colon A \otimes A \to A$ by
$f \circ g \coloneqq \smash{\sum_{k\ge0} D_k(f,g)}$.
Let the vector space $A$ equipped with this new algebra structure be
denoted~$\cA$.
By abuse of terminology, the algebra~$\cA$ is often referred to as a deformation
quantisation of~$A$.

By results of Gan and Ginzburg \cite{GG:QuantSlod} and Loseu
\cite{Los:QuantSymplWalg}, the Rees algebra of the \W-algebra considered with
the Kazhdan filtration, denoted $\Uh(\fg,e)$, is a deformation quantisation of
the ring of functions of the Slodowy slice $A = \Cpoly{\vS_\chi}$.
The \W-algebra $U(\fg,e)$ itself is then just $\cA$.

As a special case, we consider the $\Cpower{\hbar}$-extended universal
enveloping algebra $\Uh(\fg)$, which is a deformation quantisation of
$\Cpoly{\fg^*}$.
Consider the vector space $\fg_\hbar \coloneqq \fg \otimes \Cpower{\hbar}$ equipped with the
Lie bracket $[\cdot,\cdot]_\hbar$, defined as $[x,y]_\hbar \coloneqq [x,y] \hbar$ for
$x,y \in \fg$ and extended $\Cpower{\hbar}$-bilinearly.
The algebra $\Uh(\fg)$ is then the universal enveloping algebra of
$\fg_\hbar$, and can be concretely presented as the tensor algebra
$T(\fg_\hbar)$ modulo the relation $xy - yx = [x,y]_\hbar$.
This algebra $\Uh(\fg)$ is just the Rees algebra of $U(\fg)$ considered with the
{\scshape \MakeLowercase{PBW}} filtration.

Assume now that $G$ is an algebraic group which acts on $\cAh$ by
$\Cpower{\hbar}$\nb-algebra automorphisms, and preserves the grading.
This induces an action of $\fg$ on $\cAh$ by derivations, and we denote the
derivation corresponding to $\xi \in \fg$ by $\xi_\cA$.
Let there furthermore exist a \emph{quantum comoment map} for the action
of~$G$ on $\cAh$, i.e.~a linear map $\Phi \colon \fg \to \cAh$, which is
$G$\nb-equivariant and satisfies
$\smash{\tfrac{1}{\hbar}} \liebr{\Phi(\xi),\cdot} = \xi_\cA$.
It shall be useful to extend this $\Cpower{\hbar}$-linearly to a map
$\Phi \colon \Uh(\fg) \to \cAh$.

\begin{defn}
Let $\cAh$ be a deformation quantisation on which $G$ acts with quantum
comoment map $\Phi$.
Let $\gamma \in \fg^*$ be fixed under the co-adjoint action of~$G$, and define $\cI_\gamma$
as the two-sided ideal in $\Uh(\fg)$ generated by
$\fg_{\hbar,\gamma} \coloneqq \set{x - \gamma(x)\hbar \st x \in \fg}$.
The \emph{quantum Hamiltonian reduction} of $\cAh$ at $\gamma$ under the action
of~$G$ is
\begin{equation*}
\cAh \qqq{\gamma} G \coloneqq \pp[\big]{\cAh \big/ \Acoset{\gamma}}^G.
\end{equation*}
This has a natural algebra structure with multiplication given by
\begin{equation*}
\pp{a + \Acoset{\gamma}}\pp{b + \Acoset{\gamma}} = ab + \Acoset{\gamma}.
\end{equation*}
\end{defn}

\begin{rem}
Let $\cAh$ be a deformation quantisation of the Poisson variety $X$, and
let $G$ act on~$\cAh$ with quantum comoment map $\Phi$.
Assume further that the action of~$G$ on~$\cAh$ is induced by an action of~
$G$ on~$X$.
Then the quantum comoment map $\Phi$ induces a classical moment map
$\mu \colon X \to \fg^*$, and for any $\gamma \in \fg^*$ fixed under the co-adjoint
action of $G$, the quantum Hamiltonian reduction $\cAh \qqq{\gamma} G$ is a
deformation quantisation of the classical Hamiltonian reduction $X \qq{\gamma} G$.
\end{rem}

We shall now give a justification for calling \cref{def:Walg} the definition of
the \W-algebra by quantum Hamiltonian reduction.
Choosing a Premet subalgebra $\fm$ for~$e$ naturally produces an algebraic
group~$M \subseteq G$ by exponentiation, since $\fm$ is an ad-nilpotent subalgebra
(\labelcref{PremNil}).
This acts on $\fg^*$ by the restriction of the co-adjoint action, and on
$\Uh(\fg)$ by extending the adjoint action.
Furthermore, this action has a quantum comoment map
$\Phi \colon \fm \to \Uh(\fg)$ which comes from the natural inclusion of~$\fm$
into~$\fg$, and extends to the natural inclusion of $\Uh(\fm)$ into
$\Uh(\fg)$.

Since $\chi \in \fm^*$ is a character of $\fm$ (\labelcref{PremChar}) it is fixed
under the co-adjoint action of~$M$, and so we can consider the quantum
Hamiltonian reduction $\Uh(\fg) \qqq{\chi} M$.
Since $M$ is a unipotent algebraic group, invariants under adjoint action of $M$
are completely equivalent to invariants under the adjoint action of $\fm$.
As a result,
\begin{equation*}
\Uh(\fg) \qqq{\chi} M
 \coloneqq \pp[\big]{\Uh(\fg) \big/ \Uh(\fg)\Phi(\cI_\chi)}^M
 = \pp[\big]{\Uh(\fg) \big/ \Uh(\fg)\cI_\chi}^\fm,
\end{equation*}
and passing through the projection $\hbar \mapsto 1$ results in the definition of
the \W-algebra given in \cref{def:Walg}.
We can therefore without ambiguity denote the above by $U(\fg) \qqq{\chi} \fm$.

\subsection{Quantum Hamiltonian reduction by stages}
\label{sec:ReductionStages}

Consider the Slodowy slice associated to the zero nilpotent,
$\vS_0 = \fg^*$.
\Cref*{eq:SlodowyRed} can then be restated in the following way: the Slodowy
slice~$\vS_\chi$ can be expressed as a Hamiltonian reduction of $\vS_0$.
In order to answer for which other pairs of nilpotent elements $e_1$ and~$e_2$
this can be done, we need to introduce the machinery of \emph{Hamiltonian
reduction by stages}.

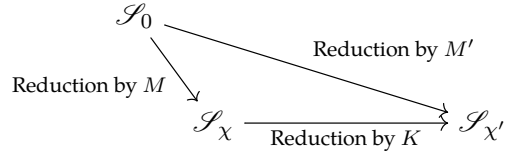
\begin{figure}[ht]
\centering

\begin{tikzcd}[column sep=tiny]
\vS_0
  \ar{dr}[swap]{\text{Reduction by } M}
  \ar{drrr}{\text{Reduction by } M'} \\
& \vS_\chi \ar{rr}[dashed, swap]{\text{Reduction by } K}
& \hspace{5em} & \vS_{\chi'}
\end{tikzcd}
\caption{Hamiltonian reduction of Slodowy slices by stages}
\label{eq:SlodowyReductions}
\end{figure}

Reduction by stages is a technique for decomposing a Hamiltonian reduction into
a sequence of smaller reductions.
The general theory is quite highly developed (cf.~\cite{MMO:HamRed}), but we
shall be interested in the specific case of reduction by a semidirect product.

Let $G \simeq H \rtimes K$ be an algebraic group which is a semidirect product of the
closed subgroups $H$ and~$K$, with $H$ normal in~$G$.
Let $X$ be a Poisson variety with a Hamiltonian action of $G$ with equivariant
moment map $\mu \colon X \to \fg^*$.
Let $\gamma \in \fg^*$ be a regular value of~$\mu$, which is identified with $(\eta,\kappa)$
under the decomposition $\fg^* = \fh^* \times \fk^*$.
Under certain mild conditions on the subgroup $K$ and the values $\eta$
and~$\kappa$, there exists an isomorphism of Poisson varieties
\begin{equation*}
X \qq{\gamma} G \simeq \pp{X \qq{\eta} H} \qq{\kappa} K,
\end{equation*}
where all the induced actions are well-defined and Hamiltonian.

With this in mind, we seek to define an analogous  notion of \emph{quantum
Hamiltonian reduction by stages}.

\begin{thm}
\label{thm:2StageReduction}
Let $\cAh$ be a deformation quantisation, and let $G \simeq H \rtimes K$ be an
algebraic group which acts on it with quantum comoment map
$\Phi \colon \fg \to \cAh$.
Let $\gamma \in \fg^*$ be an invariant under the co-adjoint action of $G$, which
decomposes as $\gamma = (\eta,\kappa)$ under the identification
$\fg^* \simeq \fh^* \times \fk^*$.
Then there exists a natural action of $K$ on $\cAh \qqq{\eta} H$ with an
induced quantum comoment map $\Phi_K \colon \fk \to \cAh \qqq{\eta} H$, and there
is a natural homomorphism of algebras
\begin{equation*}
\pp[\big]{\cAh \qqq{\eta} H} \qqq{\kappa} K \to \cAh \qqq{\gamma} G.
\end{equation*}
\end{thm}

\begin{proof}
We first show that the reduced spaces are properly defined and there exists an
action of $K$ on $\cAh \qqq{\eta} H$ with a quantum comoment map denoted~$\Phi_K$.
First note that restricting the action of $G$ yields an action of $H$ on $\cAh$,
and the restriction $\Phi_H \coloneqq \res{\Phi}{\fh} \colon \fh \to \cAh$ is
$H$\nb-equivariant due to $G$\nb-equivariance and the normality of~$H$.
Further, $\eta$ is fixed by $G$ and hence by $H$, so the reduction
$\cAh \qqq{\eta} H$ is well-defined.

We next define the action of $K$ on $\cAh \qqq{\eta}H$.
Recall that
\begin{equation}
\cAh \qqq{\eta} H
  = \set[\big]{a + \Acoset[H]{\eta} \st \Ad_h (a) \in a + \Acoset[H]{\eta}
  \;\;\forall \; h \in H}.
\label{eq:QHR1stage}
\end{equation}
Define the action of $K$ by
$\Ad_k \pp[\big]{a + \Acoset[H]{\eta}} \coloneqq \Ad_k(a) + \Acoset[H]{\eta}$.
This is independent of the choice of representative, as can be seen by the
following calculation, taking $k \in K$ and $x \in \fh$:
\begin{align*}
\Ad_k \Phi(x - \eta(x)\hbar)
  & = \Phi( \Ad_k x - \eta(x)\hbar )
    && \text{$\Phi$ is $G$- and so $K$\nb-equivariant} \\
  & = \Phi( \Ad_k x - \Ad^*_k(\eta)(x)\hbar )
    && \eta \in \smash{(\fh^*)^G} \subseteq \smash{(\fh^*)^K} \\
  & = \Phi ( \Ad_k x - \eta(\Ad_k x)\hbar ) \in \Phi(\cI_\eta)
    && H \trianglelefteq G \text{ and so } \Ad_k x \in \fh
\end{align*}
That $\Ad_k \pp[\big]{a + \Acoset[H]{\eta}}$ remains $H$\nb-invariant again follows
from the normality of~$H$, as
$\Ad_h \Ad_k (\overline{a})
  = \Ad_{hk} (\overline{a})
  = \Ad_{kh'} (\overline{a})
  = \Ad_k \Ad_{h'} (\overline{a})
  = \Ad_k (\overline{a})$.

\parpic(10em,13em)[r]{

\begin{tikzcd}[row sep=scriptsize, column sep=small]
\fh \ar[swap, hookrightarrow]{dd} \ar{ddr}{\Phi_H}                                  \\
\phantom{A}\\
\fg \ar[two heads, xshift=.5ex]{dd}
      \ar[yshift=.6ex]{r}[inner sep=.2ex]{\Phi_{\hphantom{\eta}}}
      \ar[yshift=-.6ex]{r}[swap]{\Phi_\eta}
  & \cAh \ar[two heads]{d}{\pi_H}      \\
  & \cAh \big/ \Acoset[H]{\eta}                                      \\
\fk \ar[hookrightarrow, xshift=-.5ex]{uu} \ar{r}{\Phi_K}
  & \cAh \qqq{\eta} H  \ar[hookrightarrow]{u}[swap]{\iota_H}
\end{tikzcd}
}
Lastly, we need to exhibit a quantum comoment map
$\Phi_K \colon \fk \to \cAh \qqq{\eta} H$.
We first define an $\eta$-twisted quantum comoment map, extending $\eta$ by zero on
$\fk$ and defining
\begin{align*}
\Phi_\eta \colon \fg & = \fh \rtimes \fk \to \cAh  && \\
\Phi_\eta(x) & \coloneqq \Phi(x) - \eta(x)\hbar                  &&
\end{align*}
To define the quantum comoment map $\Phi_K$, consider $\fk$ as the quotient
$\fg/\fh$, choose an arbitrary lift from $\fk$ to $\fg$, and apply the function
$\pi_H \circ \Phi_\eta$, where $\pi_H$ is the projection
$\pi_H \colon \cAh \to \cAh / \Acoset[H]{\eta}$.
To see this is well-defined, note that for any $y \in \fk$ and $x \in \fh$,
\begin{equation*}
\Phi_K(y) = \Phi_\eta(y + x) = \Phi(y) + \Phi(x - \eta(x)\hbar) \in \Phi(y) + \Acoset[H]{\eta}.
\end{equation*}
To see that the image of $\Phi_K$ lies within $H$\nb-invariants, note that the
semidirect product structure of $\fg$ guarantees that for $y \in \fk$ and $h \in H$
there exists an $x \in \fh$ such that $\Ad_h y = y + x$.
From this we can see that
\begin{equation*}
\Ad_h \Phi_K(y) = \Phi_\eta(\Ad_h y)
  = \Phi(y) + \Phi(x - \eta(x)\hbar) \in \Phi_K(y) + \Acoset[H]{\eta}.
\end{equation*}
That this map is $K$\nb-equivariant follows directly from the
$G$\nb-equivariance of~$\Phi$, and the quantum comoment condition
$\tfrac{1}{\hbar}\liebr{\Phi_K(\xi),\cdot} = \xi_{\cAh \qqq{\eta} H}$
follows from the corresponding condition for $\Phi$ and the above calculations
showing the action of $K$ is well-defined on $\cAh \qqq{\eta} H$.
Since $\kappa \in \fk^*$ is fixed by the action of~$G$, and hence by~$K$, we can
therefore talk sensibly about the two-stage reduction
$\pp{\cAh \qqq{\eta} H} \qqq{\kappa} K$.
Furthermore, the action of~$K$ on $\cAh \qqq{\eta} H$ descends to an action on
$\pp{ \cAh \qqq{\eta} H } \big/ \pp{ \cAh \qqq{\eta} H } \Phi_K(\cI_\kappa)$.

\begin{figure}[ht]
\centering

\begin{tikzcd}[column sep=-3em]
\cAh                   \ar[two heads]{dr}{\pi_H}
                       \ar[two heads, bend left]{dddrrr}{\pi_G}           \\
  & \cAh \big/ \Acoset[H]{\eta}   \ar[two heads]{ddrr}{\pi_G^H}              \\
\cAh \qqq{\eta} H         \ar[two heads]{dr}{\pi_K} \ar[hookrightarrow]{ur}{\iota_H}   \\
  & \pp{\cAh \qqq{\eta} H} \big/ \pp{\cAh \qqq{\eta} H} \Phi_K(\cI_\kappa)
                       \ar[dashed]{rr}{\widetilde{\varphi}}
    &\hspace{9em}& \cAh \big/ \Acoset{\gamma}                                \\
\pp{\cAh \qqq{\eta} H} \qqq{\kappa} K                \ar[hookrightarrow]{ur}{\iota_K}
                       \ar[dashed]{rr}{\varphi}
    && \cAh \qqq{\gamma} G  \ar[dashed]{ul}[swap]{\psi} \ar[hookrightarrow]{ur}{\iota_G}
\end{tikzcd}
\caption{Quantum Hamiltonian reduction by stages}
\label{fig:QHRstages}
\end{figure}
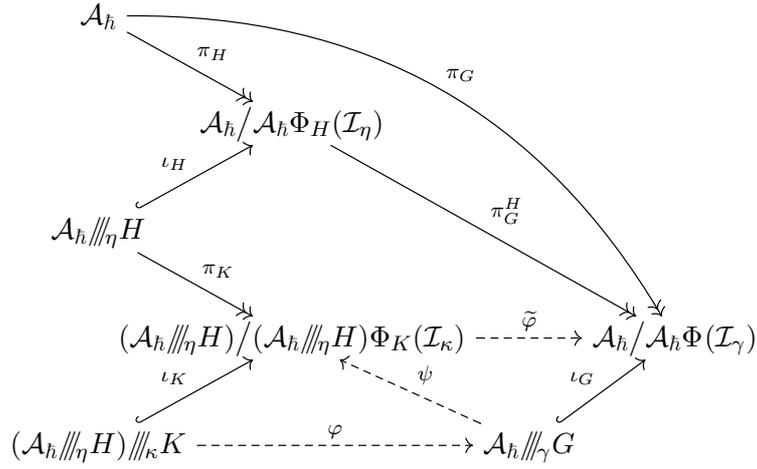

It remains to show that there is a map from the two-stage reduction
$\pp{\cAh \qqq{\eta} H} \qqq{\kappa} K$ to the one-shot reduction $\cAh \qqq{\gamma} G$.
To this end we shall construct the map~$\varphi$ from \cref{fig:QHRstages}.
Consider the maps
\begin{align*}
\widetilde{\varphi} & \colon
  \pp{\cAh \qqq{\eta} H} \big/ \pp{ \cAh \qqq{\eta} H} \Phi_K(\cI_\kappa)
    \to \cAh \big/ \Acoset{\gamma}                           \\
\varphi & \colon
  \pp{\cAh \qqq{\eta} H} \qqq{\kappa} K
    \to \cAh \qqq{\gamma} G,
\end{align*}
where $\widetilde{\varphi}$ is defined by first lifting to $\cAh \qqq{\eta} H$ and then
applying $\pi_G^H \circ \iota_H$, and~$\varphi$ is defined as the composition
$\widetilde{\varphi} \circ \iota_K$.
To show these are well-defined requires checking that~$\widetilde{\varphi}$ doesn't
depend on the lift chosen, and that the image of~$\varphi$ lies in $G$\nb-invariants.
These can both be checked by careful but straightforward calculations.
We have therefore constructed a homomorphism~$\varphi$ from the two-stage to the
one-shot reduction; it is merely the identity map suitably interpreted in the
appropriate cosets: \begin{equation*}
\varphi \pp[\big]{(a + \Acoset[H]{\eta}) + (\cAh \qqq{\eta} H) \Phi_K(\cI_\kappa)}
  = a + \Acoset{\gamma}.\qedhere
\end{equation*}
\end{proof}

\begin{cor}
\label{cor:2StageReductionIso}
If this homomorphism $\varphi$ induces an isomorphism
$\overline{\varphi} \colon \pp[\big]{A \qq {\eta} H} \qq{\kappa} K \isoto A \qq{\gamma} G$
of the corresponding Poisson algebras, then it is itself an isomorphism.
In particular, this holds if $\cAh$ is a deformation quantisation of a Poisson
manifold for which the classical reduction by stages hypotheses hold.
\end{cor}

\begin{proof}
The first statement follows from the fact that $\varphi$ induces a homomorphism of
Poisson algebras, and from the fact that $A$ and $\cA$ have identical underlying
vector spaces.
That this induces an isomorphism if $\cAh$ is a deformation quantisation of a
Poisson manifold is just the classical Hamiltonian reduction by stages
construction, which can be found in e.g.~\cite[§5.3]{MMO:HamRed}.
\end{proof}

\begin{eg}
Consider the algebra $\Uh(\fsl_3)$, which is a deformation quantisation of the
Poisson variety $\fsl_3^*$.
By restriction of the adjoint action of $SL_3$, $\Uh(\fsl_3)$ is acted on by the
following groups:
\begin{align*}
N & = \set*{ \begin{psmallmatrix} 1&0&0 \\ r&1&0 \\ t&s&1 \end{psmallmatrix}
              \st r,s,t \in \CC },
& M & = \set*{ \begin{psmallmatrix} 1&0&0 \\ r&1&0 \\ t&0&1 \end{psmallmatrix}
              \st r,t \in \CC },
& K & = \set*{ \begin{psmallmatrix} 1&0&0 \\ 0&1&0 \\ 0&s&1 \end{psmallmatrix}
              \st s \in \CC }.
\end{align*}
Note that $N = M \rtimes K$, and the quantum comoment maps associated to the
actions are given by the inclusions of their respective Lie algebras into
$\Uh(\fsl_3)$.
Let $\chi \in \fn^*$ be the character corresponding to the regular nilpotent element
$e = \begin{psmallmatrix} 0&1&0 \\ 0&0&1 \\ 0&0&0 \end{psmallmatrix}$ under the
Killing isomorphism, and let $\eta \in \fm^*$ and $\kappa \in \fk^*$ be the restrictions
of~$\chi$.

The quantum Hamiltonian reduction $\Uh(\fsl_3) \qqq{\chi} N$ is the polynomial ring
$\Cpoly{z_1,z_2}$, where
\begin{align*}
z_1 & \coloneqq h_1^2 + h_1 h_2 + h_2^2 + 3\hbar(e_1 + e_2) \\
z_2 & \coloneqq 2h_1^3 + 3h_1^2 h_2 - 3h_1 h_2^2 - 2h_2^3
          + 9\hbar e_1(h_1 + 2h_2) - 9\hbar e_2(2h_1 + h_2)
          + 27\hbar^2(e_3 + e_2).
\end{align*}
These can be lifted to invariants under the action of $M$ in
$\Uh(\fsl_3) \big/ \Uh(\fsl_3)\Phi_M(\cI_{\eta})$ as
\begin{align*}
z_1 & \mapsto z_1 + 3e_2 (f_2 - \hbar) &
z_2 & \mapsto z_2 + 9(e_2h_2 + \hbar e_3 - \hbar e_2) (f_2 - \hbar).
\end{align*}
Passing to the quotient
$\pp[\big]{\Uh(\fsl_3) \qqq{\eta} M} \big/ \pp[\big]{\Uh(\fsl_3) \qqq{\eta} M} \Phi_K(\cI_\kappa)$
yields the well-defined map~$\psi$ of \cref{fig:QHRstages}, and its image lies in
$K$\nb-invariants.
\end{eg}


\section{Reduction by stages for W-algebras in type~A}
\label{chap:TypeA}

Given the framework of quantum Hamiltonian reduction by stages, we can now try
to find an explicit realisation in the case of \W-algebras, as in
\cref{fig:IntRedDiag}.
From this point we shall work over $\CC$ and in type~A, assuming that
$\fg = \fsl_n$. 
In this case, we have a simple classification of both the conjugacy classes of
nilpotent elements and of their good gradings.

Recall the set of nilpotent orbits in $\fsl_n$ is parameterised by partitions
of~$n$, corresponding to the sizes of the Jordan blocks for the nilpotent.
The set of nilpotent orbits in a Lie algebra always has a natural partial
ordering, where $\vO_{e_1} \le \vO_{e_2}$ is defined to mean
$\vO_{e_1} \subseteq \smash{\overline{\vO}_{e_2}}$.
In type~A, this coincides with the dominance ordering on partitions:
take $\lambda = (\lambda_1, \dotsc, \lambda_k)$ and $\mu = (\mu_1, \dotsc, \mu_k)$, where $\lambda$ and $\mu$
are padded on the right with zeros if necessary, and define $\lambda \le \mu$ to mean that
$\smash{\cramped{\sum_{i=1}^\ell}} \lambda_i \le \smash{\cramped{\sum_{i=1}^\ell}} \mu_i$ for
every $\ell = 1, \dotsc, k$.
A classical theorem of Gerstenhaber classifies the covering relations in the
dominance ordering, and roughly corresponds to `sliding a box up' in the
corresponding Young diagram.

\begin{lem}{\normalfont [Gerstenhaber]}
\label{prop:OrbitCover}
The partition $\lambda$ covers $\mu$ if and only if there exist indices $j < k$ with
$\mu_j = \lambda_j - 1$, $\mu_k = \lambda_k + 1$ and $\lambda_i = \mu_i$ otherwise, where $j$ is the
smallest index such that $0 \le \lambda_k < \lambda_j - 1$ and either $k = j+1$ or
$\lambda_k = \lambda_j - 2$.
\end{lem}

\begin{eg}
\label{eg:OrbitCover}
The partitions $\lambda = (3)$, $\mu = (2,1)$ and $\nu = (1,1,1)$ cover one another in
turn.
\begin{align*}
\ytableausetup{centertableaux,boxsize=1.1em}
\lambda & = \ydiagram{3}     &
\mu & = \ydiagram{2,1}   &
\nu & = \ydiagram{1,1,1}
\ytableausetup{nocentertableaux,boxsize=\boxsize}
\end{align*}
\end{eg}

\subsection{Pyramids}
\label{sec:TypeAPyramids}

The problem of classifying all good gradings has been solved by Elashvili
and Kac \cite{EK:ClassGG};
in the classical types, this is accomplished using a combinatorial structure
known as a \emph{pyramid}.
In type~A, pyramids are an enriched version of Young diagrams, allowing
horizontal shifts of the rows according to certain conditions.
In this paper we shall use the French convention for Young diagrams.

\begin{defn}
\label{def:typeApyramid}
Let $\lambda = (\lambda_1, \dotsc, \lambda_k)$ be a partition of $n$.
A \emph{pyramid of shape $\lambda$} is a Young diagram of shape $\lambda$ consisting of
boxes of size~2, along with integer horizontal row shifts such that
the co-ordinates of the first (resp.\ last) boxes in each row form an
increasing (resp.\ decreasing) sequence.

A \emph{filling} of a pyramid is a labelling of each of the boxes with a number
between~1 and~$n$, such that there are no repeated labels.
Given a filled pyramid, the column and row of the box labelled~$k$ are denoted
$\col(k)$ and $\row(k)$, respectively.
We say $\ell$ is \emph{right-adjacent} to $k$, denoted $k \to \ell$, if
$\row(k) = \row(\ell)$ and $\col(k) + 2 = \col(\ell)$.
\end{defn}

\begin{note}
The row and column of a box are only well-defined up to an integer shift.
However, since we'll only ever be concerned with differences of row and column
numbers, this will not cause a problem.

When filling pyramids, we shall most often choose the labelling so that it
increases first up columns and then left to right.
\end{note}

\begin{eg}
The three pyramids of shape $(4,3)$ follow, each with a sample filling.
\begin{subequations}
\label{eg:pyramids}
\begin{minipage}{.333\textwidth}
  \begin{equation}
  \label{eg:pyramidLeft}
  \raisebox{-1.5em}{
  \begin{picture}(80,40)
  \put(0,0) {\line(1,0){80}} \put(0,20){\line(1,0){80}}
  \put(0,40){\line(1,0){60}}
  \put(0,0) {\line(0,1){40}} \put(20,0){\line(0,1){40}}
  \put(40,0){\line(0,1){40}} \put(60,0){\line(0,1){40}}
  \put(80,0){\line(0,1){20}}
  \put(10,10){\makebox(0,0){{1}}}
  \put(10,30){\makebox(0,0){{2}}}
  \put(30,10){\makebox(0,0){{3}}}
  \put(30,30){\makebox(0,0){{4}}}
  \put(50,10){\makebox(0,0){{5}}}
  \put(50,30){\makebox(0,0){{6}}}
  \put(70,10){\makebox(0,0){{7}}}
  \end{picture}
  }
  \end{equation}
\end{minipage}
\begin{minipage}{.333\textwidth}
  \begin{equation}
  \label{eg:pyramidDynkin}
  \raisebox{-1.5em}{
  \begin{picture}(80,40)
  \put(0,0)  {\line(1,0){80}} \put(0,20) {\line(1,0){80}}
  \put(10,40){\line(1,0){60}}
  \put(0,0)  {\line(0,1){20}} \put(20,0) {\line(0,1){20}}
  \put(40,0) {\line(0,1){20}} \put(60,0) {\line(0,1){20}}
  \put(80,0) {\line(0,1){20}}
  \put(10,20){\line(0,1){20}} \put(30,20){\line(0,1){20}}
  \put(50,20){\line(0,1){20}} \put(70,20){\line(0,1){20}}
  \put(10,10){\makebox(0,0){{1}}}
  \put(20,30){\makebox(0,0){{2}}}
  \put(30,10){\makebox(0,0){{3}}}
  \put(40,30){\makebox(0,0){{4}}}
  \put(50,10){\makebox(0,0){{5}}}
  \put(60,30){\makebox(0,0){{6}}}
  \put(70,10){\makebox(0,0){{7}}}
  \end{picture}
  }
  \end{equation}
\end{minipage}
\begin{minipage}{.333\textwidth}
  \begin{equation}
  \label{eg:pyramidRight}
  \raisebox{-1.5em}{
  \begin{picture}(80,40)
  \put(0,0) {\line(1,0){80}} \put(0,20){\line(1,0){80}}
  \put(20,40){\line(1,0){60}}
  \put(0,0) {\line(0,1){20}} \put(20,0) {\line(0,1){40}}
  \put(40,0) {\line(0,1){40}} \put(60,0) {\line(0,1){40}}
  \put(80,0) {\line(0,1){40}}
  \put(10,10){\makebox(0,0){{1}}}
  \put(30,10){\makebox(0,0){{2}}}
  \put(30,30){\makebox(0,0){{3}}}
  \put(50,10){\makebox(0,0){{4}}}
  \put(50,30){\makebox(0,0){{5}}}
  \put(70,10){\makebox(0,0){{6}}}
  \put(70,30){\makebox(0,0){{7}}}
  \end{picture}
  }
  \end{equation}
\end{minipage}
\end{subequations}
\end{eg}

\begin{thm}{\normalfont \cite[Theorem~4.2]{EK:ClassGG}}
There is a bijection between the pyramids of size $n$ and the set of good
gradings in $\fsl_n$ up to conjugacy. The same holds in $\fgl_n$.
\end{thm}

Consider a filled pyramid $P$.
The nilpotent element $e_P$ associated to $P$ is just the nilpotent element
associated to $P$ considered as a Young tableau, namely
$\smash{\sum_{i \to j}} E_{ij}$.
The grading $\Gamma_P$ associated to $P$ is defined by declaring $E_{ij}$ to be of
graded degree $\col(j) - \col(i)$.
It can be checked that this grading is good for $e_P$.



\subsection{Reduction by stages for W-algebras}
\label{sec:SlodowyReduction}

In this section we shall use the machinery of pyramids to produce a quantum
Hamiltonian reduction by stages for \W-algebras in type~A.
Since our reductions are by nilpotent groups, it will suffice to work with the
Lie algebras, which completely determine the actions of the corresponding
algebraic groups.

\begin{obj}
\label{thm:SlodowyReductionStages}
Let $\fg = \fsl_n$, and $e_1, e_2 \in \fg$ be two nilpotent elements such
that $\vO_{e_1} < \vO_{e_2}$.
We would like to construct an algebraic group $K$ with a quantum Hamiltonian
action on $U(\fg,e_1)$, along with a character $\kappa \in \fk^*$, such that
$U(\fg,e_2) \simeq U(\fg,e_1) \qqq{\kappa} K$.
\end{obj}

It will suffice to produce such a construction for every pair such that
$\vO_2$ covers~$\vO_1$.
For any such pair we will chose nilpotent elements $e_i \in \vO_i$ with
respective duals $\chi_i \in \fg^*$ for $i = 1,2$, a good grading $\Gamma_1$ for~$e_1$
with a Premet subalgebra~$\fm_1$, and a subalgebra $\fm_2 \supseteq \fm_1$ satisfying:
\begin{enumerate}[label=\textsf{SR\arabic*}.,ref=\textsf{SR\arabic*},
                  labelindent=\parindent,leftmargin=*]
\item \label[condition]{cond:SubalgDecomp}
      $\fm_2$ decomposes as a semidirect product $\fm_2 = \fm_1 \rtimes \fk$.
\item \label[condition]{cond:SubalgCharacter}
      $\chi_2$ restricts to a character of $\fm_2$, and $\chi_2 = (\chi_1, \kappa)$ in the
      above decomposition.
\item \label[condition]{cond:KStabilise}
      the subalgebra $\fk$ annihilates $\chi_1$.
\end{enumerate}

Since $\fm_1$ is a Premet subalgebra for $e_1$, the corresponding quantum
Hamiltonian reduction by stages will ensure that
\begin{equation*}
U(\fg) \qqq{\chi_2} \fm_2 \simeq \pp[\big]{U(\fg) \qqq{\chi_1} \fm_1} \qqq{\kappa} \fk =
  U(\fg,e_1) \qqq{\kappa} \fk.
\end{equation*}
We will therefore provide a construction satisfying these conditions, and
conjecture that the algebra $U(\fg) \qqq{\chi_2} \fm_2$ is isomorphic to the
\W-algebra $U(\fg,e_2)$.
Provided this conjecture holds, this will accomplish
\cref*{thm:SlodowyReductionStages}.

\subsubsection{The general construction}
\label{sec:GeneralConstruction}
Let $\mu$ be the partition corresponding to the nilpotent~$e_1$.
We will construct a right-aligned pyramid for~$\mu$, i.e.~a pyramid for which the
rightmost boxes in each row all lie in the same column, and number the boxes
from bottom to top and left to right.
This determines an even good grading~$\Gamma_1$ and Premet subalgebra~$\fm_1$
for~$e_1$.
By \cref{prop:OrbitCover}, for every partition~$\lambda$ which covers~$\mu$ in the
dominance ordering, there is a pair of integers $i < j$ for which~$\lambda$ is
obtained from~$\mu$ by `sliding a box down' from row~$j$ to row~$i$.
Define $e_2$ as
\begin{equation}
e_2 \mspace{10mu} \coloneqq \mspace{10mu} e_1 \mspace{20mu} + \mspace{-25mu}
  \sum_{\substack{
    \row(k) = i, \row(\ell) = j \\
    \col(k) = \col(\ell)}}
      \mspace{-35mu} E_{k\ell},
\label{eq:GenConstructionE}
\end{equation}
and define the Lie algebras $\fm_2$ and $\fk$ by
\begin{align}
\fk   & \coloneqq \gen[\big]{ E_m }_{m=1}^{j-i} &
\text{and}\quad
\fm_2 & \coloneqq \fm_1 + \fk, &
\text{where}\quad
E_m   & \mspace{10mu} \coloneqq \mspace{-30mu}
          \sum_{\substack{
          i \leq \row(k) < \row(\ell) \leq j \\
          \row(\ell) - \row(k) = m           \\
          \col(k) = \col(\ell)}}
            \mspace{-30mu} E_{\ell k}.
\label{eq:GenConstructionM}
\end{align}

Let us further define a semisimple element $h'_2$:
\begin{equation}
h'_2 \coloneqq \mspace{-20mu}
  \sum_{\substack{
    \row(\ell)=s, \; s \neq i,j \\
    t = 0,\dotsc,\lambda_s-1 \\
    \text{$\ell$ is $t$-th from the left}}}
      \mspace{-20mu} (\lambda_s - 1 - 2t) E_{\ell \ell}
  \mspace{20mu} + \mspace{-25mu} \sum_{\substack{
    \row(\ell)=i, \; \row(m)=j \\
    t=0,\dotsc,\lambda_i \\
    \col(m) = \col(\ell)-2 \\
    \text{$\ell$ is $t$-th from the left}}}
      \mspace{-20mu} (\lambda_i + K - 2t) (E_{\ell \ell} + E_{m m}).
\label{eq:GenConstructionH}
\end{equation}
In the second term, the $E_{m m}$ term is omitted if there is no $m$ satisfying
the conditions for the given $t$, and $E_{\ell \ell}$ is omitted for $t = \lambda_i$.
Here, $K$ is the unique constant so that $h'_2$ has trace zero.
Note that $h'_2$ is a semisimple element for which $[h'_2, e_2] = 2e_2$; we
shall show that $h'_2$ determines a good grading for $e_2$ in \cref{prop:h2Good},
however $\fm_2$ is not in general a Premet subalgebra for this grading, nor does
there necessarily exist an \hsl-triple containing $e_2$ and $h'_2$.

\begin{rem}
\label{rem:NilpGens}
Note that $\fk$ is an abelian Lie algebra, and that
\begin{equation*}
\liebr{\fm_2, \fm_2} = \liebr{\fm_2, \fm_1} \subseteq
  \liebr[\big]{\bigoplus_{k\le0} \fg_k, \bigoplus_{\ell\le-2} \fg_\ell} \subseteq
  \bigoplus_{k\le-2} \fg_k = \fm_1 \subseteq \fm_2.
\end{equation*}
This confirms that $\fm_2$ is closed under the Lie bracket, and further that
$\fm_1$ is an ideal in $\fm_2$; hence, $\fm_2$ is a semi-direct product
$\fm_1 \rtimes \fk$.
\end{rem}

\begin{eg}
\label{eg:sl6middle}
Let $\fg = \fsl_6$ and consider $\mu = (2,2,2)$.
The right-aligned pyramid~$P_1$, nilpotent element~$e_1$ and Premet
subalgebra~$\fm_1$ are as follows:
\begin{align*}
P_1 & = \raisebox{1.05em}{\large \ytableaushort{36,25,14}} &
e_1 & =
\begin{pmatrix}
0 & 0 & 0 & 1 & 0 & 0\\
0 & 0 & 0 & 0 & 1 & 0\\
0 & 0 & 0 & 0 & 0 & 1\\
0 & 0 & 0 & 0 & 0 & 0\\
0 & 0 & 0 & 0 & 0 & 0\\
0 & 0 & 0 & 0 & 0 & 0
\end{pmatrix} &
\fm_1 & =
\begin{pmatrix}
0 & 0 & 0 & 0 & 0 & 0\\
0 & 0 & 0 & 0 & 0 & 0\\
0 & 0 & 0 & 0 & 0 & 0\\
* & * & * & 0 & 0 & 0\\
* & * & * & 0 & 0 & 0\\
* & * & * & 0 & 0 & 0
\end{pmatrix}
\intertext{The unique covering partition is $\lambda = (3,2,1)$, which is obtained by
`sliding a box from row~$3$ to row~$1$'.
Applying the above procedure with $i = 1$ and $j = 3$ results in}
&&
e_2 & = e_1 + E_{13} + E_{46} &
\fm_2 & = \fm_1 + \gen[\big]{E_1, E_2} \\
&&
& =
\begin{pmatrix}
0 & 0 & 1 & 1 & 0 & 0\\
0 & 0 & 0 & 0 & 1 & 0\\
0 & 0 & 0 & 0 & 0 & 1\\
0 & 0 & 0 & 0 & 0 & 1\\
0 & 0 & 0 & 0 & 0 & 0\\
0 & 0 & 0 & 0 & 0 & 0
\end{pmatrix} &
& =
\begin{pmatrix}
0 & 0 & 0 & 0 & 0 & 0\\
a & 0 & 0 & 0 & 0 & 0\\
b & a & 0 & 0 & 0 & 0\\
* & * & * & 0 & 0 & 0\\
* & * & * & a & 0 & 0\\
* & * & * & b & a & 0
\end{pmatrix}
\end{align*}
where $E_1 = E_{21} + E_{32} + E_{54} + E_{65}$ and $E_2 =  E_{31} + E_{64}$.
Further,
\begin{equation*}
h'_2 = (E_{22} - E_{55}) + (2 E_{11} + 0 E_{44} + 0 E_{33} - 2 E_{66})
\end{equation*}
\end{eg}

\subsubsection{Properties of the construction}

\begin{thm}
\label{thm:typeAred}
Under the above circumstances, $e_2$ is a nilpotent element of type~$\lambda$, $\fm_2$
is a Lie algebra and \crefrange{cond:SubalgDecomp}{cond:KStabilise} hold.
Consequently, \cref{thm:2StageReduction} holds, and so there is homomorphism
from the quantum Hamiltonian reduction by stages to the one-shot reduction:
\begin{equation*}
\pp[\big]{U(\fg) \qqq{\chi_1} \fm_1} \qqq{\kappa} \fk = U(\fg,e_1) \qqq{\kappa} \fk
  \to U(\fg) \qqq{\chi_2} \fm_2.
\end{equation*}
\end{thm}

We shall prove in \cref{thm:SlodRedByStagesIso} that this homomorphism is, in
fact, an isomorphism, but will leave this discussion until the necessary framework has been developed. 
Before proving the \namecref{thm:typeAred}, we should introduce a result of
\cite{EK:ClassGG}: given any filled pyramid $P$ with corresponding nilpotent
element $e$, the centraliser $\liez[\fg]{e}$ can be read off the pyramid~$P$.
Let $\mu = (\mu_1, \dotsc, \mu_k)$, so the $i$th row of the pyramid has $\mu_i$ boxes,
and let $b_{i,j}$ be the standard basis vector corresponding to the index of
the box in the $i$th~row, $j$th from the right in the filled pyramid.
We can represent an endomorphism of $\CC^n$ by specifying where each of the
basis vectors $b_{i,j}$ is sent in an arrow diagram.

Elashvili and Kac define a collection of endomorphisms in $\fgl_n$, denoted
$\smash{\cramped{E_i^j[r]}}$, where $i$ and~$j$ range over the rows of the
pyramid and~$r \in \NN$ varies over a range depending on $\mu_i$ and~$\mu_j$.
These endomorphisms are defined in \cref{fig:ggnilcommute}, where any basis
vector not specified is sent to zero.


\begin{figure}[ht]
\centering

\tikzstyle{dot}=[circle,fill=black,inner sep=0,minimum size=3mm]
\tikzstyle{e}=[color=gray!50]

\begin{subfigure}[b]{\textwidth}
\centering

\begin{tikzpicture}[very thick,bend angle=45,scale=.8]

  \path node (n3) at (5,0)  [dot,label=below:$b_{i,\mu_i}$] {}
        node (n4) at (7,0)  [dot,label=below:$b_{i,\mu_i-1}$] {}
          edge [->,e] (n3)
        node (d3) at (8,0)  {}
          edge [->,e] (n4)
        node (d4) at (11,0) {} edge [loosely dotted] (d3)
        node (n5) at (12,0) [dot,label=below:$b_{i,r+2}$] {}
          edge [->,e] (d4)
        node (n6) at (14,0) [dot,label=below:$b_{i,r+1}$] {}
          edge [->,e] (n5)
        node (d5) at (15,0) {}
          edge [->,e] (n6)
        node (d6) at (16,0) {} edge [loosely dotted] (d5)
        node (n7) at (17,0) [dot,label=below:$b_{i,2}$] {}
          edge [->,e] (d6)
          edge [->, bend right] (n5)
        node (n8) at (19,0) [dot,label=below:$b_{i,1}$] {}
          edge [->,e] (n7)
          edge [->, bend right] (n6);

  \begin{scope}
    \clip (n3) rectangle +(2.3,1.2);
    \draw [<-] (n3) to [bend left] +(5,0);
  \end{scope}

  \begin{scope}
    \clip (n4) rectangle +(1.5,1);
    \draw [<-] (n4) to [bend left] +(5,0);
  \end{scope}

  \begin{scope}
    \clip (n5) rectangle +(-1.5,1);
    \draw [->] (n5) to [bend right] +(-5,0);
  \end{scope}

  \begin{scope}
    \clip (n6) rectangle +(-2.3,1.2);
    \draw [->] (n6) to [bend right] +(-5,0);
  \end{scope}

\end{tikzpicture}

\caption{The endomorphism $E_i^i[r]$ on row $i$ with shift $r$, for $0 \le r < \mu_i$.
         Note that $e = \sum_i E_i^i[1]$.}
\label{fig:ggnilcommute1}
\end{subfigure}
\\
\begin{subfigure}[b]{\textwidth}
\centering

\begin{tikzpicture}[very thick,scale=.7]

  \path node (n1) at (0,0)  [dot,label=below:$b_{i,\mu_i}$] {}
        node (d1) at (1,0)  {}
          edge [->,e] (n1)
        node (d2) at (3.4,0)  {} edge [loosely dotted] (d1)
        node (n2) at (4.4,0)  [dot,label=below:$b_{i,\mu_j-r}$] {}
          edge [->,e] (d2)
        node (n3) at (7,0)  [dot,label=below:$b_{i,\mu_j-1+r}$] {}
          edge [->,e] (n2)
        node (d3) at (8,0)  {}
          edge [->,e] (n3)
        node (d4) at (9,0) {} edge [loosely dotted] (d3)
        node (n4) at (10,0)  [dot,label=below:$b_{i,2}$] {}
          edge [->,e] (d4)
        node (n5) at (12.6,0)  [dot,label=below:$b_{i,1}$] {}
          edge [->,e] (n4);

  \path node (m1) at (.7,2)  [dot,label=above:$b_{j,\mu_j}$] {}
          edge [<-] (n2)
        node (m2) at (3.3,2)  [dot,label=above:$b_{j,\mu_j-1}$] {}
          edge [->,e] (m1)
          edge [<-] (n3)
        node (c1) at (4.3,2)  {}
          edge [->,e] (m2)
        node (c2) at (5.3,2)  {} edge [loosely dotted] (c1)
        node (m3) at (6.3,2)  [dot,label=above:$b_{j,2+r}$] {}
          edge [->,e] (c2)
          edge [<-] (n4)
        node (m4) at (8.9,2)  [dot,label=above:$b_{j,1+r}$] {}
          edge [->,e] (m3)
          edge [<-] (n5)
        node (c3) at (9.9,2)  {}
          edge [->,e] (m4)
        node (c4) at (10.9,2) {} edge [loosely dotted] (c3)
        node (m5) at (11.9,2) [dot,label=above:$b_{j,1}$] {}
          edge [->,e] (c4);

\end{tikzpicture}

\caption{The endomorphism $E_i^j[r]$ from row~$i$ to row~$j$ for $i<j$ with
shift~$r$, for $0 \le r < \mu_j$.}
\label{fig:ggnilcommute3}
\end{subfigure}
\begin{subfigure}[b]{\textwidth}
\centering

\begin{tikzpicture}[very thick,scale=.7]

  \path node (n1) at (0,0)  [dot,label=below:$b_{j,\mu_j}$] {}
        node (n2) at (2.6,0)  [dot,label=below:$b_{j,\mu_j-1}$] {}
          edge [->,e] (n1)
        node (d1) at (3.6,0)  {}
          edge [->,e] (n2)
        node (d2) at (4.6,0)  {} edge [loosely dotted] (d1)
        node (n3) at (5.6,0)  [dot,label=below:$b_{j,2+r}$] {}
          edge [->,e] (d2)
        node (n4) at (8.2,0)  [dot,label=below:$b_{j,1+r}$] {}
          edge [->,e] (n3)
        node (d3) at (9.2,0)  {}
          edge [->,e] (n4)
        node (d4) at (11.6,0) {} edge [loosely dotted] (d3)
        node (n5) at (12.6,0) [dot,label=below:$b_{j,1}$] {}
          edge [->,e] (d4);

  \path node (m1) at (.7,2)  [dot,label=above:$b_{i,\mu_i}$] {}
        node (c1) at (1.7,2)  {}
          edge [->,e] (m1)
        node (c2) at (2.7,2)  {} edge [loosely dotted] (c1)
        node (m2) at (3.7,2)  [dot,label=above:$b_{i,\mu_j-r}$] {}
          edge [->,e] (c2)
          edge [->] (n1)
        node (m3) at (6.3,2)  [dot,label=above:$b_{i,\mu_j-1-r}$] {}
          edge [->,e] (m2)
          edge [->] (n2)
        node (c3) at (7.3,2)  {}
          edge [->,e] (m3)
        node (c4) at (8.3,2) {} edge [loosely dotted] (c3)
        node (m4) at (9.3,2)  [dot,label=above:$b_{i,2}$] {}
          edge [->,e] (c4)
          edge [->] (n3)
        node (m5) at (11.9,2) [dot,label=above:$b_{i,1}$] {}
          edge [->,e] (m4)
          edge [->] (n4);

\end{tikzpicture}

\caption{The endomorphism $E_i^j[r]$ from row~$i$ to row~$j$ for $i>j$ with
 shift~$r$, for $\mu_j-\mu_i \le r < \mu_j$. Note that some basis vector is always mapped
 to $b_{j,\mu_j}$.}
\label{fig:ggnilcommute2}
\end{subfigure}
\\
\caption{Endomorphisms of $\CC^n$ commuting with $e$.}
\label{fig:ggnilcommute}

\begin{note}
The nilpotent~$e$ is shown in grey in each row for reference.
\end{note}

\end{figure}
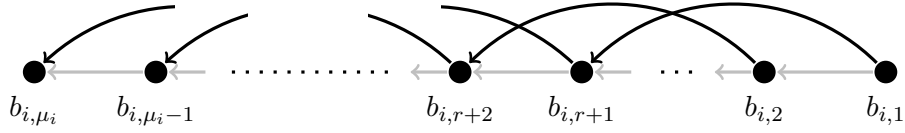
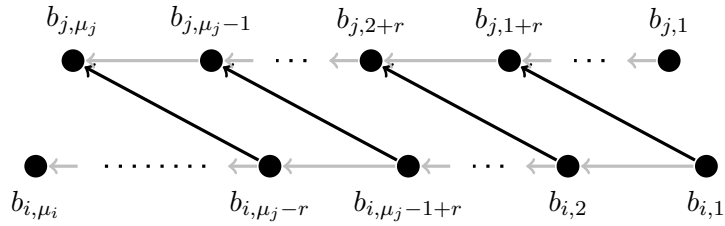
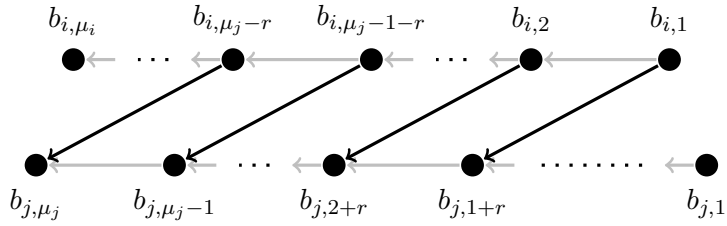

\begin{lem}{\normalfont \cite{EK:ClassGG}}
\label{lem:EKCentraliserBasis}
Let $\mu = (\mu_1,\dotsc,\mu_k)$ be a partition of $n$, and consider a filled pyramid
of shape~$\mu$ with associated nilpotent~$e$.
Then the collection $\smash{\set[\big]{E_i^j[r]}}$, where
\begin{align*}
i,j & \in \set{1, \dotsc, k} &
\text{and}\qquad\quad
\begin{aligned}
        0 & \le r < \mu_j \quad \text{if } i \le j \\
\mu_j - \mu_i & \le r < \mu_j \quad \text{if } i > j
\end{aligned}\;,
\end{align*}
forms a basis of the centraliser $\liez[{\fgl_n\mspace{-7mu}}]{e}$, and those
which lie in $\fsl_n$ form a basis of $\liez[{\fsl_n\mspace{-7mu}}]{e}$.
\end{lem}

\begin{proof}[Proof of \cref{thm:typeAred}]
To prove that $e_2$ has the correct Jordan type, it suffices to exhibit a Jordan
basis.
Note that a Jordan basis can be read off the rows of the pyramid, proceeding
from right to left.
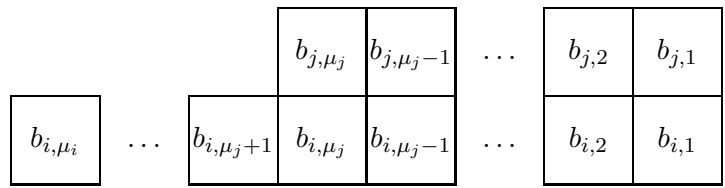
\begin{equation*}
\ytableausetup{boxsize=3em}
\begin{ytableau}
\none     & \none        & \none       & b_{j,\mu_j} & b_{j,\mu_j-1} & \none[\dots] & b_{j,2} & b_{j,1} \\
b_{i,\mu_i} & \none[\dots] & b_{i,\mu_j+1} & b_{i,\mu_j} & b_{i,\mu_j-1} & \none[\dots] & b_{i,2} & b_{i,1}
\end{ytableau}
\ytableausetup{boxsize=\boxsize}
\end{equation*}
The Jordan basis for $e_1$ in row~$i$ of this pyramid is therefore
\begin{equation*}
b_{i,\mu_i} \gets \dotsb \gets b_{i,\mu_j+1} \gets b_{i,\mu_j}
  \gets b_{i,\mu_j-1} \gets \dotsb \gets b_{i,2} \gets b_{i,1}.
\end{equation*}

The Jordan basis for $e_2$ is identical to that of $e_1$ except for those
strings corresponding to rows~$i$ and~$j$.
The Jordan basis in those rows is
\begin{gather*}
\mu_j b_{i,\mu_i} \gets \dotsb \gets \mu_j b_{i,\mu_j}
  \gets \pp[\big]{(\mu_j-1)b_{i,\mu_j-1} + b_{j,\mu_j}} \gets \dotsb
  \gets \pp[\big]{kb_{i,k} + b_{j,k+1}} \gets \dotsb
  \gets b_{j,1}, \\
\pp[\big]{b_{i,\mu_j-1} - b_{j,\mu_j}} \gets \dotsb
  \gets \pp[\big]{(\mu_j-k)b_{i,k} - b_{j,k+1}} \gets \dotsb
  \gets \pp[\big]{(\mu_j-1)b_{i,1} - b_{j,2}},
\end{gather*}
of lengths $\mu_i+1 = \lambda_i$ and $\mu_j-1 = \lambda_j$, respectively.

\Cref{cond:SubalgDecomp} is shown in \cref{rem:NilpGens}.
To check \labelcref{cond:SubalgCharacter}, note that
$\smash{\res{\chi_2}{\fm_1}} = \chi_1$ by construction, and so $\chi_2 = (\chi_1,\kappa)$ for
some $\kappa \in \fk^*$.
To confirm that $\chi_2$ is a character of~$\fm_2$, recall from \cref{rem:NilpGens}
that $\liebr{\fm_2,\fm_2} = \liebr{\fm_1 + \fk, \fm_1}$.
However, since
$\chi_2 \pp[\big]{\liebr{\fm_1,\fm_1}} = \chi_1 \pp[\big]{\liebr{\fm_1,\fm_1}} = 0$,
it remains only to check that $\chi_2(\liebr{\fk, \fm_1}) = 0$.
We shall check this on the generating set
$\set[\big]{\liebr{E_m, E_{\ell k}} \st 1 \le m \le j-i, \col(k) < \col(\ell)}$.
\begin{equation*}
\chi_2 \pp[\big]{ \liebr{E_m, E_{\ell k}} }
  = \inn[\big]{e_2}{\liebr{E_m, E_{\ell k}}}
  = \inn[\big]{\liebr{e_2, E_m}}{E_{\ell k}}
\end{equation*}
Using the language of \cref{fig:ggnilcommute}, note that
\begin{equation}
\label{eq:chiAnnihilate}
\liebr{e_2, E_m} = \liebr[\Big]{e_1 + E^i_j[0], \sum_{i\le s\le j-m} E^{s+m}_{s}[0]}
  = \sum_{i\le s\le j-m} \liebr[\Big]{E^i_j[0], E^{s+m}_s[0]} \in \fg_0.
\end{equation}
Here the second equality follows from the fact that $E^{s+m}_s[0]$ commutes
with~$e_1$;
that $\chi_2$ annihilates \cref*{eq:chiAnnihilate} now follows from \cref{GGprop5},
and the fact that $E_{\ell k} \in \bigoplus_{t<0}\fg_t$.
This further establishes the claim that $\fk$ annihilates both $\chi_2$ and $\chi_1$:
hence \cref{cond:KStabilise} also holds.
This completes the proof of \cref{thm:typeAred}.
\end{proof}

\begin{thm}
\label{thm:PremetConjProps}
The pair $e_2$ and~$\fm_2$ satisfy \crefrange{PremNil}{PremChar}, .
\end{thm}

\begin{proof}
\Cref{PremNil} is manifest from the construction, and \cref{PremChar} is a
subclaim of \cref{cond:SubalgCharacter}.
\Cref{PremDim} follows from the fact that $e_1$ itself satisfies it, along with
an application of the orbit--stabiliser theorem and
\cref{lem:EKCentraliserBasis}.

We prove \cref{PremZ} by directly calculating $\fm_2 \cap \liez[\fg]{e_2}$.
In the coming calculation, we use the following conventions:
\begin{itemize}
\item Recall that $k \to \ell$ means that $\ell$ is right-adjacent to $k$.
\item If $\row(k) = i$, then $\up{k}{p}$ indicates that $p$ is the box such that
      $\row(p) = j$ and $\col(p) = \col(k)$, if such exists.
      Similarly, if $\row(s) = j$, then $\up{q}{s}$ indicates that~$q$ is such
      that $\row(q) = i$ and $\col(q) = \col(s)$, if such exists.
\item $A_{k\ell} = 0$ if there do not exist $k$ and $\ell$ which satisfy the
      adjacency relations specified below, and~$B_m = 0$ if $m < 1$ or $m > j - i$.
\end{itemize}
Taking the commutator of~$e_2$ with a generic element of~$\fm_2$ results in the
following:
\begin{align}
\MoveEqLeft[3] \liebr[\Big]{e_2,
  \mspace{-15mu} \sum_{\col(v) < \col(u)} \mspace{-15mu} A_{uv} E_{uv}
  + \sum_{m=1}^{j-i} B_m E_m} = \notag \\
 = {} & \sum_{\col(s) < \col(k)} \pp*{
    \Aadj[k]{\ell}{s} - \Aadj{k}{r}[s]
    + \begin{cases}
        A_{ps} & \row(k) = i, \up{k}{p} \\
        0      & \text{otherwise}
      \end{cases}
    - \begin{cases}
        A_{kq} & \row(s) = j, \up{q}{s} \\
        0      & \text{otherwise}
      \end{cases}
  } E_{ks} \notag \\
 & + \mspace{-20mu} \sum_{\col(s) = \col(k)} \pp*{
    \Aadj[k]{\ell}{s} - \Aadj{k}{r}[s]
    + \begin{cases}
        B_m  & \row(k) = i, \row(s) = j-m \\
        -B_m & \row(s) = j, \row(k) = i+m \\
        0    & \text{otherwise}
      \end{cases}
  } E_{ks}. \label{eq:PremZ}
\end{align}
For \cref{eq:PremZ} to vanish, we will prove that all of $A_{uv}$ and $B_m$
must vanish as well.

\begin{case}
[{$A_{uv} = 0$ for $\col(v) < \col(u)$, $\row(u) \neq i$ and $\row(v) \neq j$}]
\hspace{0pt}\\
\label{case:PremZGen}
\indent
Examining the coefficient of $E_{uw}$ for $v \to w$ yields $A_{tw} - A_{uv}$,
where $u \to t$.
We can prove the claim by induction on the distance of $u$ from the right of the
pyramid.
The base case is when~$u$ is rightmost in its row:
then $A_{tw} = 0$, and so $A_{uv} = 0$.
The same argument assuming the inductive hypothesis for all~$u$ within $n$~boxes
of the right proves the claim for all $u$ within $n+1$ boxes of the right,
completing the induction.
\end{case}

\begin{case}
[{$A_{uv} = 0$ for $\col(v) < \col(u)-1$, $\row(u) = i$ and $\row(v) \neq j$}]
\hspace{0pt}\\
\label{case:PremZi}
\indent
Examining the coefficient of $E_{uw}$ for $v \to w$ yields
$A_{tw} - A_{uv} + A_{pw}$, where $u \to t$ and $\up{u}{p}$.
But $A_{pw} = 0$ by \cref{case:PremZGen}, and so the same argument as above
completes the case.
\end{case}

\begin{case}
[{$A_{uv} = 0$ for $\col(v) < \col(u)-1$, $\row(u) \neq i$ and $\row(v) = j$}]
\hspace{0pt}\\
\label{case:PremZj}
\indent
Apply the argument of \cref{case:PremZi} \emph{mutatis mutandis}.
\end{case}

\begin{case}
[{$A_{uv} = 0$ for $\col(v) < \col(u)-1$, $\row(u) = i$ and $\row(v) = j$}]
\hspace{0pt}\\
\label{case:PremZij}
\indent
The conclusions of \cref{case:PremZi,case:PremZj} allow the same argument to
again be applied \emph{mutatis mutandis}.
\end{case}

\begin{case}
[{$A_{uv} = 0$ for $\col(v) = \col(u)-1$}]
\hspace{0pt}\\
\label{case:PremZm}
\indent
If neither $\row(u) = i$ or $\row(v) = j$ this is dealt with by
\cref{case:PremZGen}, while if both of these hold there is no contribution from
any $B_m$.
Since the arguments are symmetric, we'll assume that $\row(u) = i$.

Assume that $\row(v) = j-m$ for $1 \le m \le j-i$; otherwise there is no
contribution from any $B_m$ and we're done.
Since $e_2$ covers $e_1$ there are exactly the same number of boxes in the two
rows; let the boxes of row~$i$ be labelled from the left $u_1, \dotsc, u_k$ and
the boxes of row~$j$ be labelled $v_1, \dotsc, v_k$.
The sum of the co-efficients of $E_{u_1 v_1}$ up to $E_{u_k v_k}$ is $k B_m$,
which proves that $B_m = 0$.
The argument of \cref{case:PremZGen} proves that the remaining $A_{uv}$ must
vanish, which completes this last case and the proof of the theorem. \qedhere
\end{case}
\end{proof}

Finally, recalling the element $h'_2$ from \cref{eq:GenConstructionH}, we
establish the following result for future reference.
\begin{lem}
The grading coming from the semisimple element $h'_2$ is good for~$e_2$.
Further, $\ad h'_2$ preserves $\fm_1$.
\label{prop:h2Good}
\end{lem}

\begin{proof}
By construction we have that $[h'_2, e_2] = 2e_2$, so all that remains is to
show that $\bigoplus_{i<0} \fg_i \cap \liez[\fg]{e_2} = \set{0}$, which is equivalent
to \cref{GGprop2}.
Considering the basis of the centraliser $\liez[\fg]{e_1}$ given in
\cref{lem:EKCentraliserBasis}, note that the basis of the centraliser
$\liez[\fg]{e_2}$ is closely related:
it differs only in that endomorphisms which involve basis vectors in rows~$i$
and~$j$ have these replaced by an appropriate linear combination of basis
vectors as in the proof of \cref{lem:EKCentraliserBasis}.
However, these linear combinations lie in the zero weight space of $h'_2$ by
construction;
this proves that any element of the centraliser cannot lie in strictly negative
degree.
That $\ad h'_2$ preserves $\fm_1$ follows immediately from the fact that $h'_2$
is diagonal.
\end{proof}

\subsection{Relation to W-algebras}

\begin{conj}
\label{conj:typeAred}
For nilpotents $e_1, e_2 \in \fg$ and subalgebras $\fm_1, \fm_2 \subseteq \fg$ as defined
in \cref{eq:GenConstructionE,eq:GenConstructionM}, the reduction by stages
$U(\fg,e_1) \qqq{\kappa} \fk \simeq U(\fg) \qqq{\chi_2} \fm_2$ is isomorphic to the
\W-algebra $U(\fg,e_2)$.
\end{conj}

\begin{rem}
This conjecture is a special case of a more general conjecture due to Premet,
based on his work in \cite{Pre:TransSlice}.
Specifically, Premet conjectures that for any pair of subalgebra $\fm$ and
nilpotent $e$ which satisfy \crefrange{PremNil}{PremChar}, the quantum
Hamiltonian reduction $U(\fg) \qqq{\chi} \fm$ is isomorphic to the \W-algebra
$U(\fg,e)$.
In fact, Premet has proven this conjecture in the case that the base field is of
non-zero characteristic~$p$.
\end{rem}

\begin{prop}
\label{prop:typeAsubreg}
\Cref{conj:typeAred} holds for $e_1$ a subregular and $e_2$ a regular nilpotent.
\end{prop}

\begin{proof}
The subalgebra $\fm_2$ constructed is simply the maximal nilpotent subalgebra of
lower-triangular matrices $\fn^-$.
This is a Premet subalgebra for $e_2$.
\end{proof}

\begin{rem}
The construction detailed in this \namecref{sec:GeneralConstruction} can be
modified slightly to give a stronger version of \cref*{prop:typeAsubreg}.
Instead of choosing a right-aligned pyramid of shape $\mu$, one can choose a
pyramid which is right-aligned but for a leftward shift of 1 at row~$i$ and
another leftward shift of 1 at row~$j+1$.
This necessitates a choice of Lagrangian $\fl \subseteq \fg_{-1}$;
this choice can be made so that the resulting Premet subalgebra can be extended
to a Premet subalgebra for a pyramid of shape~$\lambda$,
which is right-aligned but for a leftward shift of 1 at row~$i+1$ and another
leftward shift of 2 at row~$j$.
\begin{equation*}
\begin{picture}(200,113)(0,-13)
\put(0,0)  {\line(1,0){60}} \put(0,20) {\line(1,0){60}}
\put(10,40){\line(1,0){40}} \put(10,60){\line(1,0){40}}
\put(10,80){\line(1,0){40}} \put(20,100){\line(1,0){20}}
\put(0,0)  {\line(0,1){20}} \put(20,0) {\line(0,1){20}}
\put(40,0) {\line(0,1){20}} \put(60,0) {\line(0,1){20}}
\put(10,20){\line(0,1){60}} \put(30,20){\line(0,1){60}}
\put(50,20){\line(0,1){60}}
\put(20,80){\line(0,1){20}} \put(40,80){\line(0,1){20}}
 \put(10,10){\makebox(0,0){{1}}}
 \put(30,10){\makebox(0,0){{5}}}
 \put(50,10){\makebox(0,0){{10}}}
 \put(20,30){\makebox(0,0){{2}}}
 \put(40,30){\makebox(0,0){{7}}}
 \put(20,50){\makebox(0,0){{3}}}
 \put(40,50){\makebox(0,0){{8}}}
 \put(20,70){\makebox(0,0){{4}}}
 \put(40,70){\makebox(0,0){{9}}}
 \put(30,90){\makebox(0,0){{6}}}
 \put(30,-13){\makebox(0,0){{$\mu = (3,2,2,2,1)$}}}

\put(100,40){\makebox(0,0){{<}}}

\put(140,0)  {\line(1,0){60}} \put(140,20) {\line(1,0){60}}
\put(140,40) {\line(1,0){60}} \put(150,60){\line(1,0){40}}
\put(150,80){\line(1,0){20}} \put(150,100){\line(1,0){20}}
\put(140,0)  {\line(0,1){40}} \put(160,0) {\line(0,1){40}}
\put(180,0) {\line(0,1){40}} \put(200,0) {\line(0,1){40}}
\put(150,40){\line(0,1){60}} \put(170,40){\line(0,1){60}}
\put(190,40){\line(0,1){20}}
 \put(150,10){\makebox(0,0){{1}}}
 \put(170,10){\makebox(0,0){{5}}}
 \put(190,10){\makebox(0,0){{10}}}
 \put(150,30){\makebox(0,0){{2}}}
 \put(170,30){\makebox(0,0){{7}}}
 \put(190,30){\makebox(0,0){{9}}}
 \put(160,50){\makebox(0,0){{3}}}
 \put(180,50){\makebox(0,0){{8}}}
 \put(160,70){\makebox(0,0){{4}}}
 \put(160,90){\makebox(0,0){{6}}}
 \put(170,-13){\makebox(0,0){{$\lambda = (3,3,2,1,1)$}}}
\end{picture}
\label{eq:AlternateConstruction}
\end{equation*}

For this new pyramid and compatible choice of Lagrangian, \cref{thm:typeAred}
remains true.
Furthermore, \cref{prop:typeAsubreg} and its proof hold not only for $e_1$ a
subregular and $e_2$ a regular nilpotent, but more generally for any pair of
nilpotent elements $e_1$ and $e_2$ of types $\mu = (\mu_1,\dotsc,\mu_k,1)$ and
$\lambda = (\mu_1,\dotsc,\mu_k+1)$, respectively.
\end{rem}

\begin{eg}
Consider $\fsl_4$, and $e_1$ a nilpotent of type $(2,2)$.
This is covered by the subregular nilpotent~$e_2$, and so the construction will
produce an algebra $U(\fg) \qqq{\chi_2} \fm_2$.
The associated graded algebra of $U(\fg, e_2)$ is the ring of functions on the
Slodowy slice $\vS_{\chi_2}$, and the associated graded algebra of the reduced
space is $\Cpoly[\big]{\fg^* \qq{\chi_2} M_2}$.
\begin{gather*}
\begin{aligned}
\vS_{\chi_2}
  & = \set*{
      \begin{pmatrix}
        a       & 1      & 0 & 0 \\
        b-3a^2  & a      & 1 & 0 \\
        c+20a^3 & b-3a^2 & a & d \\
        f       & 0      & 0 & -3a
      \end{pmatrix}
      \st a,b,c,d,f \in \CC} \\
\fg^* \qq{\chi_2} M_2
  & \simeq \set*{
      \begin{pmatrix}
        0                  & 1              & 1                 & 0 \\
        x + \frac{u+v}{4}  & 0              & 0                 & 1 \\
        \frac{-3u+v}{4}    & -2y            & 0                 & 1 \\
        z + \frac{u+v}{2}y & \frac{u-3v}{4} & x + \frac{u+v}{4} & 0
      \end{pmatrix}
      \st u,v,x,y,z \in \CC}
\end{aligned} \\
\begin{aligned}
\Cpoly[\big]{\vS_{\chi_2}}
  & = \Cpoly{a,b,c,d,f}, &
\Cpoly[\big]{\fg^* \qq{\chi_2} M_2}
  & = \Cpoly{u,v,x,y,z}
\end{aligned}
\end{gather*}
These are isomorphic as Poisson algebras, as shall be shown below.

Consider the ring homomorphism
$\varphi \colon \Cpoly[\big]{\vS_{\chi_2}} \to \Cpoly[\big]{\fg^* \qq{\chi_2} M_2}$
defined by
\begin{align*}
\varphi(a) & = \tfrac{-1}{3} y   & \varphi(b) & = x   & \varphi(c) & = 2z - \tfrac{8}{3} xy   &
\varphi(d) & = v + x + y^2       & \varphi(f) & = -u - x - y^2.
\end{align*}
The non-zero Poisson brackets are given by the formulae:
\begin{align*}
\poibr{a,d} & = \tfrac{-1}{24} d     & \poibr{c,d} & = \tfrac{1}{6} bd  & &&
  \poibr{u,y} & = \tfrac{1}{8}  (u + x + y^2) &
  \poibr{u,z} & = \tfrac{1}{4} x (u + x + y^2) \\
\poibr{a,f} & = \tfrac{1}{24} f      & \poibr{c,f} & = \tfrac{-1}{6} bf & &&
  \poibr{v,y} & = \tfrac{-1}{8} (v + x + y^2) &
  \poibr{v,z} & = \tfrac{-1}{4} x (v + x + y^2) \\
\poibr{d,f} & = \mathrlap{\tfrac{-27}{2} a^3 + ab - \tfrac{1}{8} c} & & & &&
  \poibr{u,v} & = \mathrlap{\tfrac{-1}{4} \pp[\big]{ z + xy + 2(u+v)y }}.
\end{align*}
It can be checked that this map is a ring isomorphism and also preserves the
Poisson bracket; it hence induces an isomorphism of the Poisson varieties
$\vS_{\chi_2} \simeq \fg^* \qq{\chi_2} M_2$.
Furthermore, this map preserves the characteristic polynomial.

Thus, we know that our algebra $U(\fg) \qqq{\chi_2} \fm_2$ is a deformation
quantisation of the Slodowy slice $\vS_{\chi_2}$.
Since there is, up to isomorphism, a unique such deformation quantisation which
is $G$\nb-equivariant, and the \W-algebra $U(\fg,e_2)$ is such a deformation
quantisation, it follows that $U(\fg) \qqq{\chi_2} \fm_2 \simeq U(\fg,e_1) \qqq{\kappa} \fk$
is isomorphic to $U(\fg,e_2)$.
\end{eg}

\section{The representation theory of W-algebras}
\label{chap:CatOWalg}


The construction of quantum Hamiltonian reduction by stages has a number of
applications to the representation theory of \W-algebras.
\Cref{thm:typeAred} has an immediate corollary relating the categories of
modules over $U(\fg,e_1)$ and $U(\fg,e_2)$.

\begin{cor}
Let $e_1$ and $e_2$ be two nilpotent elements of $\fsl_n$ such that
$e_2$ covers~$e_1$ in the dominance ordering;
then the quantum Hamiltonian reduction by stages construction produces an
adjoint pair of functors
$\Mod{U(\fg,e_1)} \leftrightarrows \Mod{\pp[\big]{U(\fg) \qqq{\chi_2} \fm_2}}$.
If \cref{conj:typeAred} holds, then there exists an adjunction
$\Mod{U(\fg,e_1)} \leftrightarrows \Mod{U(\fg,e_2)}$ for any pair of nilpotents
$e_2 \ge e_1$.
\end{cor}

\begin{proof}
Note that the quotient $U(\fg,e_1) / U(\fg,e_1)\fk_\kappa$ is a
$\pp[\big]{U(\fg,e_1), U(\fg) \qqq{\chi_2} \fm_2}$\nb-bimodule, where the left module
structure comes from left multiplication by $U(\fg,e_1)$ and the right module
structure comes from the fact that
$U(\fg) \qqq{\chi_2} \fm_2 \simeq \smash{\cramped{\pp[\big]{U(\fg,e_1) / \fk_\kappa}^\fk}}$.
This proves the existence of the first adjunction.

If \cref{conj:typeAred} holds, then the latter algebra is isomorphic to
$U(\fg,e_2)$.
Since adjunctions can be composed, and their composition is itself an
adjunction, we can form such an adjunction for any pair of nilpotents
$e_2 \ge e_1$ by composing along a sequence of covering relations.
\end{proof}

There are also applications to the \W-algebraic analogue of the
{\scshape \MakeLowercase{BGG}} category~$\cO$:
a full subcategory of $\Mod{U(\fg,e)}$ whose definition we shall recall in the
next \namecref{sec:CatOWalg}, and which has been studied in,
e.g.~\cite{BGK:HWTWAlg,Web:CatOWAlg}.
In \cite{Los:CatOWAlg}, Loseu investigates its structure and constructs an
equivalence between it and a certain subcategory of Whittaker modules in
$\Mod{U(\fg)}$.
The objective of this \namecref{chap:CatOWalg} is to prove a similar
equivalence in type~A, relating the \W-algebraic categories~$\cO$ for different
nilpotents to one another.
In what follows we shall always assume \cref{conj:typeAred} holds.

\subsection{Categories~\texorpdfstring{$\mathbfcal{O}$}{O} and other related
categories for W-algebras}

\label{sec:CatOWalg}

To discuss the categories~$\cO$ for the \W-algebra $U(\fg,e)$, we need to fix a
choice of parabolic subalgebra $\fp \subseteq \fg$ such that $(e,h,f)$ is contained in
the Levi subalgebra $\fl \subseteq \fp$.
Further, we shall fix a maximal torus $\ft$ of the centraliser $\liez[\fg]{e}$,
noting that $\ft \subseteq \fl$.

In place of the choice of parabolic, Loseu instead chooses a cocharacter~$\theta$
of~$T$ viewed as an element of~$\ft$;
this uniquely determines a parabolic as the positive eigenspaces of $\theta$.
Different choices of~$\theta$ will only matter inasmuch as they determine different
parabolics, so as far as we are concerned the two points of view are equivalent.

This choice of parabolic and maximal torus allows us to define a pre-order on
the weights of $\ft$:
$\lambda \ge \mu$ if and only if $\lambda - \mu$ is a linear combination of the weights of $\ft$
acting on $\fp$.
The existence of an embedding $U(\ft) \hookrightarrow U(\fg,e)$
(cf.~\cite[Theorem~3.3]{BGK:HWTWAlg}) allows any $U(\fg,e)$\nb-module to be
decomposed into generalised weight spaces with respect to $\ft$.

Note also that, as proven by Premet, $Z(\fg,e) \coloneqq Z(U(\fg,e))$ is isomorphic to
the ordinary centre of the universal enveloping algebra $U(\fg)$, and that the
natural map $Z(\fg) \to U(\fg,e)$ is an isomorphism onto the centre.
Thus, central characters of $U(\fg)$ can be translated to central characters of
$U(\fg,e)$.

This allows for a number of different full subcategories of $\Mod{U(\fg,e)}$ to
be defined, the objects of which satisfy various subsets of the following
conditions.
\begin{enumerate}[
  label=($\cO$\arabic*),
  ref=$\cO$\arabic*]
\item \label[condition]{cond:wOroof}
    The $\ft$\nb-weights are contained in a finite union of sets of the form
    $\set{\mu \st \mu \le \lambda}$.

\item \label[condition]{cond:wOfd}
    The generalised weight spaces with respect to $\ft$ are finite-dimensional.

\item \label[condition]{cond:wOtss}
    The action of $\ft$ on the module is semisimple.


\item \label[condition]{cond:wOZss}
    The action of $Z(\fg,e) \coloneqq Z(U(\fg,e))$ on the module is semisimple.
\end{enumerate}
The notation used for these categories differs amongst different papers.
We will mostly keep to a pared-down version of the notation used in Webster's
papers \cite{Web:CatOWAlg}, but since the machinery and proof of Loseu's work
\cite{Los:CatOWAlg} is extremely important here, we shall present it as well.
Loseu's notation leaves the nilpotent~$e$ implicit, which would render it
ambiguous in the context of this paper.
\begin{table}[h]
\centering
\begin{tabular}{cccccccc}
\toprule
Conditions
    & 1     
    & 1,2   
    & 1,3   
    & 1,2,3 
    & 1,2,4 
    \\
\midrule
Notation
    & $\widetilde{\cO}(e,\fp)$
    & $\widehat{\cO}(e,\fp)$
    &
    & $\cO(e,\fp)$
    & $\cO'\mspace{-4mu}(e,\fp)$ \\
Webster
    &
    & $\widehat{\cO}(\mathcal{W}_e,\fp)$
    &
    & $\cO(\mathcal{W}_e,\fp)$
    & $\cO'\mspace{-4mu}(\mathcal{W}_e,\fp)$ \\
Loseu
    & $\widetilde{\cO}(\theta)$
    & $\cO(\theta)$
    & $\widetilde{\cO}^\ft\mspace{-2mu}(\theta)$
    & $\cO^\ft\mspace{-2mu}(\theta)$
    & \\
\bottomrule
\end{tabular}
\caption{Definitions of the W-algebraic categories~$\cO$}
\begin{note}
The full subcategories on which the centre $Z(\fg,e)$ acts by a given
generalised central character $\xi$ are denoted $\smash{\widehat{\cO}(\xi,e,\fp)}$,
and so on.
\end{note}
\label{tab:WalgCatO}
\end{table}

\begin{rem}
There are a number of equivalent ways of phrasing the above conditions, some of
which are used in Loseu's original paper.

\noindent
\Cref{cond:wOroof} is equivalent to:
\begin{enumerate}
\item[{($\cO$1')}] $U(\fg,e)_{>0}$ acts by locally nilpotent endomorphisms.
\end{enumerate}
Further, \cref{cond:wOroof,cond:wOfd} together are equivalent to
\cref{cond:wOroof} and:
\begin{enumerate}
\item[{($\cO$2')}] The $U(\fg,e)^0$\nb-module obtained after taking
            $U(\fg,e)_{>0}$\nb-invariants is of finite dimension, where
            $U(\fg,e)^0 \coloneqq U(\fg,e)_{\ge0} / (U(\fg,e)U(\fg,e)_{>0} \cap U(\fg,e)_{\ge0}$.
\end{enumerate}
\end{rem}

In addition, there are a number of full subcategories of $\Mod{U(\fg)}$ of
interest to us; these are variations on subcategories of modules known as
\emph{generalised Whittaker modules}.
To define these categories, we need to fix a choice of maximal nilpotent
subalgebra $\fn$ along with a character $\chi \colon \fn \to \CC$;
we can then define the shifted Lie algebra $\fn_\chi \coloneqq \set{\xi - \chi(\xi) \st \xi \in \fn}$.
We will again consider full subcategories whose objects satisfy various subsets
of the following conditions.
\begin{enumerate}[
  label=($\mathsfup{Wh}$\arabic*),
  ref=$\mathsfup{Wh}$\arabic*]
\item \label[cond]{cond:Whitroof}
    The shifted Lie algebra $\fn_\chi$ acts by locally nilpotent endomorphisms.
\item \label[cond]{cond:WhitZlf}
    The action of the centre $Z(\fg)$ is locally finite.
\item \label[cond]{cond:Whittss}
    The action of $\ft$ on the module is semisimple.
\item \label[cond]{cond:WhitZss}
    The action of $Z(\fg)$ on the module is semisimple.
\end{enumerate}
\begin{table}[h]
\centering
\begin{tabular}{cccccccc}
\toprule
Conditions
    & 1     
    & 1,2   
    & 1,3   
    & 1,2,3 
    & 1,2,4 
      \\
\midrule
Notation
    & $\smash{\Whit(U(\fg), \fn_\chi)}$
    & 
    &
    & 
    & 
      \\[0.1em]
Webster
    &
    & $\smash{\bigoplus_\xi} \widehat{\cO}(\xi,\chi)$
    &
    & $\smash{\bigoplus_\xi} \cO(\xi,\chi)$
    & $\smash{\bigoplus_\xi} \cO'\mspace{-4mu}(\xi,\chi)$
      \\[0.1em]
Loseu
    & $\smash{\widetilde{\operatorname{Wh}}}(e,\theta)$
    & $\operatorname{Wh}(e,\theta)$
    & $\smash{\widetilde{\operatorname{Wh}}^{\raisebox{-0.3em}{$\scriptstyle \ft$}}}
        \mspace{-2mu}(e,\theta)$
    & $\smash{\operatorname{Wh}^{\ft}}\mspace{-2mu}(e,\theta)$
    & \\
\bottomrule
\end{tabular}
\caption{Definitions of the Whittaker categories}
\begin{note}
Here $\xi$ ranges over the set of generalised central characters of $U(\fg)$.
\end{note}
\label{tab:WhittakerCats}
\end{table}

%

\begin{thm}
[{\cite[Theorem~4.1]{Los:CatOWAlg},
  \cite[Theorem~1.2.2(iii)]{Los:QuantSymplWalg},
  \cite[Proposition~7]{Web:CatOWAlg}}]
There are equivalences between each of columns of \cref{tab:WalgCatO} and the
corresponding columns of \cref{tab:WhittakerCats}.
These equivalences still hold if one restricts to a given generalised character
of $Z(\fg)$ and the corresponding character of $Z(\fg,e)$.
\label{thm:CatOWalgEquiv}
\end{thm}

We shall provide a version of this theorem for which the Whittaker categories
lie not in $\Mod{U(\fg)}$, but rather in $\Mod{U(\fg,e')}$ for another nilpotent
$e' \le e$.

\subsection{Equivariant Slodowy slices}
\label{sec:EqvtSlodowy}

The remaining results of this paper 
follow the techniques and methodology presented in Loseu's papers
\cite{Los:QuantSymplWalg,Los:FinRepWAlg,Los:CatOWAlg}, but translated to the
context of this paper.
We present them here for clarity and to highlight the changes necessary in our
situation.

Given a nilpotent element $e$ with a good grading~$\Gamma$ given by the semisimple
element $h'$, construct a $\Gamma$\nb-graded \hsl-triple $(e,h,f)$.
Based on this data, Loseu defines the \emph{equivariant Slodowy slice}:
\begin{equation}
\vS~_\chi \coloneqq G \times \vS_\chi \subseteq G \times \fg^* \simeq T^*G.
\end{equation}
This is a symplectic subvariety of the cotangent bundle $T^*G$, and is stable
under a number of group actions.
The group $G$ acts on itself both on the left and right by multiplication, which
this induces corresponding Hamiltonian actions on the cotangent bundle.
Loseu defines the following group actions:
\begin{itemize}
\item $G$ acts by $g \cdot (g_1, \alpha) = (g g_1, \alpha)$.
\item $\CC*$ acts by $t \cdot (g_1, \alpha) = (g_1 \gamma(t)^{-1}, t^{-2} \gamma(t) \alpha)$.
\end{itemize}
Here, $\gamma \colon \CC* \to G$ is the cocharacter determined by exponentiation of
the semisimple element~$h'$.
Recall also that $G$ acts on $\fg^*$ by the coadjoint action
$(g \alpha)(\xi) = \alpha(\Ad_{g^{-1}} \xi)$.

Choosing a Premet subgroup $M$ for $e$, we further define an action of
$M$ on $T^*G$:
\begin{itemize}
\item $M$ acts by $m \cdot (g_1, \alpha) = (g_1 m^{-1}, m \alpha)$.
\end{itemize}
This has moment map $\mu \colon G \times \fg^* \to \fm^*$ given by
$\mu(g,\alpha) = \smash{\res{\alpha}{\fm}}$.
It is therefore clear that $T^*G \qq{\chi} M \simeq \vS~_\chi$, where this is
the usual symplectic Hamiltonian reduction.

Having translated the problem of Hamiltonian reduction of Slodowy slices as
Poisson varieties to that of Hamiltonian reduction of equivariant Slodowy slices
as symplectic varieties, we can now state the following
\namecref{thm:SlodRedByStagesIso}.

\begin{thm}
\label{thm:SlodRedByStagesIso}
The homomorphism $U(\fg,e_1) \qqq{\kappa} \fk \to U(\fg) \qqq{\chi_2} \fm_2$ of
\cref{thm:typeAred} is an isomorphism.
\end{thm}

\begin{proof}
After taking $G$\nb-invariants of this symplectic reduction, we obtain the
previous Poisson reduction of Slodowy slices.
By \cref{cor:2StageReductionIso}, it therefore suffices to prove that the
quantum Hamiltonian reduction induces an isomorphism of the classical
Hamiltonian symplectic reductions.
However this follows by classical symplectic Hamiltonian reduction by stages for
semidirect products, which can be found in,
e.g.~\cite[Theorem~4.2.2]{MMO:HamRed}.
\end{proof}

Recall now the constructions of \cref{chap:TypeA}.
For any $e_1 \in \fsl_n$ and any nilpotent orbit $\vO_2$ which covers the orbit of
$e_1$, we produce a Premet subalgebra $\fm_1$, a nilpotent $e_2 \in \vO_2$, a
subalgebra $\fm_2$, and a semisimple element~$h'_2$ which gives a grading~$\Gamma$
which is good for~$e_2$.
Choosing $(e_2, h_2, f_2)$ to be a $\Gamma$\nb-graded \hsl-triple, we can define
the Slodowy slice $\vS_{\chi_2}$.
The reduction by stages construction of \cref{chap:TypeA} therefore produces the
following commutative diagram:
\begin{equation}
\begin{tikzcd}[]
T^*G                    \rar[-,double equal sign distance]
  & T^*G                                                                \\
G \times \pp{\chi_2 + \fm_2^{\bot}}     \uar[hookrightarrow] \rar[hookrightarrow]{\iota}
  & G \times \pp{\chi_1 + \fm_1^{\bot}} \uar[hookrightarrow]                            \\
\vS~_{\chi_2}           \uar[twoheadleftarrow,xshift=.60ex]
                            \uar[hookrightarrow,xshift=-.60ex]
                            \rar[hookrightarrow]{\varphi}
  & \vS~_{\chi_1}       \uar[twoheadleftarrow,xshift=.60ex]
                            \uar[hookrightarrow,xshift=-.60ex]
\end{tikzcd}
\label{eq:EqvtSlodowyInclusions}
\end{equation}
Here, the vertical maps
$G \times \smash{\pp{\chi_i + \fm_i^{\bot}}} \hookrightarrow T^*G$
and
$G \times \smash{\pp{\chi_i + \fm_i^{\bot}}} \twoheadrightarrow
  \smash{\vS~_{\chi_i}}$
are the natural maps coming from Hamiltonian reduction, while the inclusions of
$\vS~_{\chi_i}$ come from the natural presentation
$\smash{\vS_{\chi_i} = \chi_i + \pp{\fg/[\fg,f_i]}^* \subseteq \chi_i + \fm_i^\bot}$.
The map $\iota$ is the natural extension of the inclusion
$\chi_2 + \fm_2^\bot \hookrightarrow \chi_1 + \fm_1^\bot$,
and $\varphi$ is defined as the obvious composition of maps.

In this context, we shall consider the following additional actions on $T^*G$,
which preserve each of $\smash{\vS~_{\chi_i}}$:
\begin{itemize}
\item $Q \coloneqq Z_G(e_1,h_1,f_1) \cap Z_G(e_2,h_2,f_2) \cap Z_G(h'_2)$ acts by
        $g_0 \cdot (g_1, \alpha) = (g_1 g_0^{-1}, g_0 \alpha)$.
\item $\widetilde{G} \coloneqq G \times \CC* \times Q$, which acts component-wise.
\end{itemize}

\begin{lem}
The map $\varphi$ is a $\widetilde{G}$\nb-equivariant embedding of symplectic
manifolds.
\end{lem}

\begin{proof}
That $\varphi$ is $(G \times Q)$\nb-equivariant is manifest from the construction.
To see that $\varphi$ is injective, note that $\varphi(x) = \varphi(y)$ if and only if $\iota(x)$
and $\iota(y)$ lie in the same $M_1$\nb-orbit.
But then $x$ and $y$ lie in the same $M_2$\nb-orbit, as $M_1 \subseteq M_2$, which would
imply that $x = y$.

To see that $\varphi$ is symplectic, note that the $M_i$\nb-orbits form a nilfoliation
of $\smash{\vS~_{\chi_i}}$ in $\smash{G \times \pp{\chi_i + \fm_i^\bot}}$.
Hence, lifting along the $M_2$\nb-orbits can equally well be accomplished by
restricting to lifting along $M_1$\nb-orbits, and so the symplectic forms
will agree.

As in the classical case, the cocharacter $\gamma$ associated to the semisimple
element $h'_2$ gives an action of $\CC*$ on $\smash{\vS~_{\chi_2}}$.
By \cref{prop:h2Good}, the adjoint action $\Ad_{\gamma(t)}$ stabilises $\fm_1$, and
hence this also gives a well-defined action on $\smash{\vS~_{\chi_1}}$.
The map $\varphi$ intertwines these two actions, and in both cases scales the
symplectic form: $t \cdot \omega = t^2 \omega$.
\end{proof}

Putting these facts together yields the following
\namecref{thm:SlodowyEmbedding}.

\begin{thm}
\label{thm:SlodowyEmbedding}
For any pair of nilpotent elements $e_1 \le e_2$ in $\fsl_n$ in the dominance
ordering, there is a $(G \times Q)$\nb-equivariant embedding of symplectic
manifolds $\smash{\vS~_{\chi_2} \hookrightarrow \vS~_{\chi_1}}$.
Taking $G$\nb-invariants yields an embedding of Poisson manifolds
$\smash{\vS_{\chi_2} \hookrightarrow \vS_{\chi_1}}$.
Furthermore, there exist $\CC*$\nb-actions on both sides, intertwined by the
embedding, which scale the symplectic forms (resp.~Poisson bivectors) by a
factor of $t^2$.
This $\CC*$\nb-action is a contracting action on $\vS_{\chi_2}$.
\end{thm}

\subsection{The decomposition lemma}
\label{sec:DecompLem}

Consider the point $x = (1,\chi_2)$; the embedding of \cref{thm:SlodowyEmbedding}
induces an inclusion of symplectic vector spaces
$\varphi_* \colon T_x \smash{\vS~_{\chi_2}} \hookrightarrow T_x \smash{\vS~_{\chi_1}}$.
This, in turn, induces an inclusion
$\pp{\fg / [\fg, f_2]}^* \hookrightarrow \pp{\fg / [\fg, f_1]}^*$,
and we let $W$ denote its image.
Finally, we define the subspace $V \subseteq \liez[\fg]{e_1}$ as follows:
\begin{equation}
V \coloneqq W^\bot \cap \liez[\fg]{e_1} = \set{ \xi \in \liez[\fg]{e_1}
  \st \alpha(\xi) = 0 \text{ for all } \alpha \in W}.
\label{eq:Vdefn}
\end{equation}
This vector space has a symplectic form expressed by $\omega(x,y) = \chi_2([x,y])$.
Note that if $e_1 = 0$, then $V = [\fg, f_2]$.

The symplectic form on $T_x \vS~_{\chi_1}$ is given by the expression
\begin{equation*}
\omega(\xi + \alpha, \eta + \beta) = \chi_2([\xi,\eta]) - \dpair{\xi,\beta} + \dpair{\eta,\alpha},
\end{equation*}
so the symplectic complement $\pp[\big]{\varphi_* T_x \smash{\vS~_{\chi_2}}}^{\bot_\omega}$ is
$\set{(\xi, \ad_\xi^* \chi_2) \st \xi \in V}$.
Projecting onto the first component identifies this with $V$, but we'll instead
identify the it with $V^*$; they are isomorphic as symplectic $Q$\nb-modules.
This gives a $(\CC* \times Q)$\nb-equivariant symplectic isomorphism
$\smash{\psi \colon T_x \vS~_{\chi_2} \oplus V^* \to T_x \vS~_{\chi_1}}$,
given by $\psi(v,w) = \varphi_*(v) + w$.

The standard considerations of Fedosov quantisation
(cf.~\cite[Section~2.2]{Los:QuantSymplWalg})
yield $(G \times Q)$\nb-invariant, homogeneous, degree~2 star products on each
of $\power{\Cpoly{\smash{\vS~_{\chi_i}}}}{\hbar}$, the Moyal-Weyl star product on
$\power{\Cpoly{V^*}}{\hbar}$, and finally on the product
$\power{\Cpoly{\smash{\vS~_{\chi_2}} \times V^*}}{\hbar}$.
Since the star products are differential, they induce star products on the
completions
$\power{\Cpoly{\smash{\vS~_{\chi_1}}}^\wedge_{Gx}}{\hbar}$ and
$\power{\Cpoly{\smash{\vS~_{\chi_2}} \times V^*}^\wedge_{Gx}}{\hbar}$.
Applying the argument of \cite[Theorem~3.3.1]{Los:QuantSymplWalg} yields a
$\widetilde{G}$\nb-equivariant $\Cpoly{\hbar}$\nb-algebra isomorphism
\begin{equation}
\Phi_\hbar \colon \power{\Cpoly{\vS~_{\chi_1}}^\wedge_{Gx}}{\hbar} \to
       \power{\Cpoly{\vS~_{\chi_2} \times V^*}^\wedge_{Gx}}{\hbar}.
\label{eq:EqvtSlodStarIso}
\end{equation}
Taking $G$\nb-invariants produces the following analogue of
\cite[Proposition~2.1]{Los:CatOWAlg}.

\begin{thm}
\label{thm:DecompLem}
There is a $(\CC* \times Q)$\nb-equivariant $\Cpoly{\hbar}$\nb-algebra
isomorphism
\begin{equation*}
\Phi_\hbar \colon \power{\Cpoly{\vS_{\chi_1}}^\wedge_{\chi_2}}{\hbar} \to
       \power{\Cpoly{\vS_{\chi_2}}^\wedge_{\chi_2}}{\hbar}
       \widehat{\otimes}_{\Cpower{\hbar}} \power{\Cpoly{V^*}^\wedge_0}{\hbar}
\end{equation*}
satisfying:
\begin{enumerate}
\item $\Phi_\hbar(\sum_{i=0}^\infty f_i \hbar^{2i})$ contains only even powers of $\hbar$.
\item The cotangent map
     $d_0(\Phi_\hbar)^* \colon \liez[\fg]{e_2} \otimes V \to \liez[\fg]{e_1}$
     co-incides with $\psi$.
\item For $\iota_1,\iota_2$ the respective embeddings of $\mathfrak{q}$ into the
      domain and codomain of $\Phi_\hbar$, then $\Phi_\hbar \circ \iota_1 = \iota_2$.
\end{enumerate}
\end{thm}

\subsection{Loseu's machinery}
\label{sec:LoseuPaper}

In order to continue, we recall the machinery Loseu has developed for proving
\cref{thm:CatOWalgEquiv}.
Let the following be given.
\begin{itemize}
\item
    $\fv = \bigoplus_{i \in \ZZ} \fv(i)$ is a graded finite-dimensional vector
    space on which a torus $T$ acting by preserving the grading.
\item
    $A \coloneqq \Sym(\fv)$, with the induced grading $A = \bigoplus_{i \in \ZZ} A_i$ and
    induced $T$-action.
\item
    $(\cA, \circ)$ is an algebra with the same underlying vector space as $A$,
    where the algebra structure comes from a $T$\nb-invariant deformation
    quantisation.
\item
    $\omega_1$ is a symplectic form on $\fv(1)$, where $\omega_1(u,v)$ is the constant term
    of the commutator in $\cA$, and $\fy$ is a lagrangian subspace of $\fv(1)$.
\item
    $\fm \coloneqq \fy \oplus \bigoplus_{i\le0} \fv(i)$.
\item
    $v_1, v_2, \dotsc, v_n$ is a homogeneous basis of $\fv$ such that $v_1, v_2,
    \dotsc, v_m$ form a basis of $\fm$.
    Further, let $d_i$ be the degree of $v_i$ and assume that
    $d_1, d_2, \dotsc, d_m$ are increasing and that all $v_i$ are
    $T$\nb-semi-invariant.
\item
    $A^\heartsuit$ is the subalgebra of $\Cpower{\fv^*}$ consisting of elements
    of the form $\sum_{i\le c} f_i$ for some $c$, where $f_i$ is a homogeneous power
    series of degree $i$.
\item
    $\cA^\heartsuit$ is the algebra $A^\heartsuit$ with multiplication as in
    $\cA$.
    Any element of $\cA^\heartsuit$ can be written as an infinite linear
    combination of monomials $\smash{v_{i_1} \circ \dotsb \circ v_{i_\ell}}$,
    where $i_1 \ge \dotsb \ge i_\ell$ and $\smash{\sum_{j=1}^\ell d_{i_j}} \le c$ for
    some $c$.
    Hence there is a filtration $F_c \cA^\heartsuit$.
\item
    $\theta$ is a co-character of $T$, and $\fv_{\ge0}$ and $\fv_{>0}$ are,
    respectively, the sums of the positive and strictly-positive
    $\ad \theta$-eigenspaces of $\fv$.
    We shall further require that $\fv_{>0} \subseteq \fm \subseteq \fv_{\ge0}$.
\item
    $\cA_{\ge0}, \cA_{>0}, \cA^\heartsuit_{\ge0}, \cA^\heartsuit_{>0}$ are all
    defined analogously.
\item
    $\cA^\wedge \coloneqq \lim \cA / \cA\fm^k$.
    Note that there is an injective algebra homomorphism
    $\cA^\heartsuit \to \cA^\wedge$.
\end{itemize}

\begin{prop}{\normalfont \cite[Proposition~5.1]{Los:CatOWAlg}}
\label{prop:los5.1}
Let $(\cA,\circ)$ and $(\cA',\circ')$ be two different algebras coming from $A$
and $\fv$ as above.
Suppose there is a subspace $\fy \subseteq \fv(1)$ which is Lagrangian for both
symplectic forms, and every element of $A$ can be written as a finite sum of
monomials in both $\cA$ and $\cA'$.
Then any homogeneous $T$\nb-equivariant isomorphism
$\Phi \colon \cA^\heartsuit \to \cA'^\heartsuit$ satisfying
$\Phi(v_i) - v_i \in F_{d_i-2} \cA + (F_{d_i} \cA \cap \fv^2 A)$
extends uniquely to a topological algebra isomorphism
$\Phi \colon \cA^\wedge \to \cA'^\wedge$ with
$\Phi(\cA^\wedge \fm) = \cA'^\wedge \fm$.
\end{prop}

\begin{cor}{\normalfont \cite[Corollary~5.2]{Los:CatOWAlg}}
\label{cor:los5.2}
The isomorphism $\Phi \colon \cA^\wedge \to \cA'^\wedge$ induces an equivalence
of categories $\Phi_* \colon \Whit(\cA,\fm) \to \Whit(\cA',\fm)$, where
$\Whit(\cA,\fm)$ is the category of $\cA^\wedge$\nb-modules which are
annihilated by some $\cA^\wedge \fm^k$.
This equivalence preserves the subcategories on which $\ft$ acts semisimply, and
commutes with the functor of taking $\fm$\nb-invariants;\
i.e.~$\Phi_*(M^\fm) = \Phi_*(M)^\fm$.
%
\end{cor}

\begin{note}
$\Whit(\cA,\fm)$ can naturally be viewed as the category of $\cA$\nb-modules on
which~$\fm$ acts by locally nilpotent endomorphisms.
This justifies the choice of notation in \cref{tab:WhittakerCats}.
\end{note}

With the constructions as before, we seek to make a set of choices which satisfy
the hypotheses of \cref{prop:los5.1}.
First, we'll fix a maximal torus $T \subseteq Q$,
and pick an arbitrary cocharacter $\theta$, viewed as an element of $\ft$.
Denote by $\fp$ the parabolic subalgebra of $\fg$ consisting of the positive
eigenspaces of $\ad \theta$.
The zero eigenspace is a Levi subalgebra $\fl \subseteq \fp$, which contains each of
$e_i$, $h_i$ and $f_i$ for $i=1,2$.
The good grading $\Gamma$ of $\fg$ induces a good grading of $\fl$, and so one can
pick a Premet subalgebra $\smash{\ul{\fm}} \subseteq \fl$ as usual, with
corresponding shift $\smash{\ul{\fm}[\chi]}$.
Let $\tfm \coloneqq \smash{\ul{\fm}} \oplus \fg_{>0}$, where $\fg_{>0}$ consists
of the strictly positive eigenspaces of the action of $\ad \theta$;
the corresponding shift is $\tfm_\chi \coloneqq \smash{\ul{\fm}[\chi] \oplus \fg_{>0}}$.

\begin{enumerate}
\item Define $\fv \coloneqq \set{\xi - \chi_2(\xi) \st \xi \in \liez{e_1}}$.
     Note that $\fv \simeq \liez[\fg]{e_2} \oplus V$, as shown in \cref{sec:DecompLem}.

\item Choosing $m > 2 + 2d$, where $d$ is the maximum eigenvalue of $\ad h'_2$ on
      $\fg$, define the grading on $\fv$ to be given by
      \begin{equation*}
      \fv(i) = \set{\xi \in \fv \st (\ad h'_2-m \ad \theta)\xi = (i-2)\xi}.
      \end{equation*}

\item Set $\fm$ to be $\smash{\tfm_\chi \cap \fv}$, which satisfies
      $\fv_{>0} \subseteq \fm \subseteq \fv_{\ge0}$ by our choice of $m$.

\item Define $\cA \coloneqq U(\fg,e_1)$ and $\cA' = \Weyl{V} \otimes U(\fg,e_2)$.
\end{enumerate}

These choices satisfy the conditions beginning this \namecref{sec:LoseuPaper}.
It remains to prove the following \namecref{lem:MyLos5.1}.

\begin{lem}
\label{lem:MyLos5.1}
There is an isomorphism
$\Phi \colon U(\fg,e_1)^\heartsuit \to \pp{\Weyl{V} \otimes U(\fg,e_2)}^\heartsuit$
which satisfies the hypotheses of \cref{prop:los5.1}.
\end{lem}

\begin{proof}
It follows from the considerations of \cref{sec:DecompLem} that $\cA$ and $\cA'$
are both deformation quantisations of $A = \Sym(\fv)$.
Further, since $\cA^\heartsuit$ and $\cA'^\heartsuit$ can be identified with the
respective quotients by $\hbar - 1$ of the $\CC*$\nb-finite parts of
\begin{align*}
& \power{\Cpoly{\vS_{\chi_1}}^\wedge_{\chi_2}}{\hbar} & \text{and}\qquad &
\power{\Cpoly{\vS_{\chi_2}}^\wedge_{\chi_2}}{\hbar} \widehat{\otimes}_{\Cpower{\hbar}}
\power{\Cpoly{V^*}^\wedge_0}{\hbar},
\end{align*}
the isomorphism $\Phi_\hbar$ of \cref{thm:DecompLem} provides the necessary
map~$\Phi$.
This satisfies the hypotheses of \cref{prop:los5.1} by the same considerations
as before (cf.~\cite[Corollary~3.3.2]{Los:QuantSymplWalg}).
\end{proof}

\subsection{Category equivalences}
\label{sec:CatEquiv}

Finally, we shall use this machinery to develop the category equivalences we
need.
First, there is an equivalence
\begin{align*}
\mathcal{K'} \colon \Whit(\cA',\fm) & \to \smash{\widetilde{\cO}(e_2, \fp)},
& \mathcal{K'}(M) & \coloneqq \pp[\big]{M}^{\tfm \cap V}.
\end{align*}
To see that the image lies in $\smash{\widetilde{\cO}(e_2, \fp)}$, it suffices
to note that $U(\fg,e_2)_{>0}$ is generated by the strictly positive eigenspaces
of $\ad \theta$, all of which lie in $\fm$ by construction.
That this functor is an equivalence follows from results on representations of
Heisenberg algebras, and the fact that $\smash{\tfm \cap V}$ is a
Lagrangian subspace of $V$.

Combining this with the equivalence of \cref{cor:los5.2} yields the following
\namecref{thm:CatOEquiv}, which is the main result of this
\namecref{chap:CatOWalg} and generalises \cref{thm:CatOWalgEquiv}.

\begin{thm}
\label{thm:CatOEquiv}
There exists an equivalence of categories
\begin{equation*}
\mathcal{K} \colon \Wh(U(\fg,e_1),\fm) \to \widetilde{\cO}(e_2,\fp).
\end{equation*}
Furthermore, $\mathcal{K}$ induces the following embeddings of categories~$\cO$,
along with their block decompositions:
\begin{align*}
\widetilde{\cO}(e_2,\fp)   & \hookrightarrow \widetilde{\cO}(e_1,\fp) &
\widehat{\cO}(e_2,\fp)     & \hookrightarrow \widehat{\cO}(e_1,\fp) \\
\cO(e_2,\fp)               & \hookrightarrow \cO(e_1,\fp) &
\cO'\mspace{-4mu}(e_2,\fp) & \hookrightarrow \cO'\mspace{-4mu}(e_1,\fp).
\end{align*}
\end{thm}

\begin{proof}
The equivalence $\mathcal{K}$ is defined to be $\mathcal{K'} \circ \Phi_*$.
Let $\fm_1$ and $\fm_2$ be as in \cref{sec:SlodowyReduction},
$V_1 \coloneqq [\fg,f_1]$, and $V$ and $\fm$ be as above.
These functors can then be arranged into a commutative diagram.
\NewDocumentCommand \ots {} {\mspace{-8mu}\otimes\mspace{-3mu}}

\begin{equation}
\begin{tikzcd}[column sep = small]
\Whit(U(\fg), \pp{\tfm_1}_{\chi_1})
        \ar{r}{\sim}
  & \Whit( \Weyl{V_1} \ots U(\fg,e_1), \pp{\tfm_1}_{\chi_1} )
        \ar{r}{\sim}  
  & \widetilde{\cO}(e_1, \fp) \\
\Whit( \Weyl{V_1} \ots U(\fg,e_1), \pp{\tfm_2}_{\chi_2} )
        \ar{r}{\sim}
        \ar[hookrightarrow]{ur}
        \ar{d}{\pp{\cdot}^{\tfm_2 \cap V_1}}[swap]{\sim}
  & \Whit( \Weyl{V_1 \oplus V} \ots U(\fg,e_2), \pp{\tfm_2}_{\chi_2} )
        \ar{r}{\sim}  
        \ar{d}{\pp{\cdot}^{\tfm_2 \cap V_1}}[swap]{\sim}
  & \widetilde{\cO}(e_2, \fp) \\ 
\Whit( U(\fg,e_1), \fm)
        \ar{r}{\sim}[swap]{\Phi_*}
  & \Whit( \Weyl{V} \ots U(\fg,e_2), \fm)
        \ar{r}{\sim}[swap]{\mathcal{K'}}   
  & \widetilde{\cO}(e_2, \fp) \ar[-, double equal sign distance]{u}
\end{tikzcd}
\label{eq:CatEquivDiag}
\end{equation}

Here, the equivalences between the first two columns of each row are the
functors $\Phi_*$ of \cref{cor:los5.2} in the appropriate settings.
The functors between the second and third columns are the appropriate analogues
of the functor $\mathcal{K'}$, taking invariants with respect to
$\tfm_1 \cap V_1$, $\tfm_2 \cap (V_1 \oplus V)$ and $\tfm \cap V$, respectively.
Since these are all respective Lagrangian subspaces of $V_1$, $V_1 \oplus V$ and $V$,
it follows that they are equivalences.

From the diagram, it can be seen that there is an embedding of categories~$\cO$,
\begin{equation*}
\widetilde{\cO}(e_2,\fp) \hookrightarrow \widetilde{\cO}(e_1,\fp).
\end{equation*}
Since the functor $\mathcal{K}$ intertwines the actions of $\ft$ and $Z(\fg,e)$,
it induces an embedding of each of the above subcategories, and also their block
decompositions with respect to generalised central characters.
\end{proof}

\bibliographystyle{h-amsalphanum}
\newcommand{\etalchar}[1]{$^{#1}$}
\newcommand{\noopsort}[1]{}
\providecommand{\bysame}{\leavevmode\hbox to3em{\hrulefill}\thinspace}
\providecommand{\MR}{\relax\ifhmode\unskip\space\fi MR }
\providecommand{\MRhref}[2]{%
  \href{http://www.ams.org/mathscinet-getitem?mr=#1}{#2}
}
\providecommand{\href}[2]{#2}

\end{document}